\documentclass[11pt,a4paper,english]{article}

\usepackage[draft]{fixme}
\usepackage[bottom=1in,top=1in, left=1in, right=1in]{geometry}
\usepackage{babel}
\usepackage[numbers]{natbib}
\usepackage{amsthm}
\usepackage{amssymb} %
\usepackage{amsmath}
\usepackage{amsfonts}
\usepackage{verbatim}
\usepackage{centernot} %
\usepackage{stmaryrd} %
\usepackage{mathrsfs} %
\usepackage{paralist} %
\usepackage{tikz}
\usepackage{subfig}
\usetikzlibrary{arrows,shapes}
\usepackage{enumerate}
\usepackage{cancel}
\usepackage[matrix,arrow,curve]{xy}
\usepackage{graphicx}

\DeclareMathOperator{\Con}{\mathrm{Con}}
\DeclareMathOperator*{\lcm}{lcm}

\renewcommand{\emptyset}{\varnothing}

\newcommand{\<}{\langle}
\renewcommand{\>}{\rangle}

\newcommand{\keywords}[1]{{\renewcommand{\thefootnote}{\relax}\footnotetext{\emph{Keywords:}
    #1}}}

\renewcommand{\o}{\vee}

\newcommand{\y}{\wedge}

\newcommand{\ent}{\Rightarrow}
\newcommand{\tne}{\Leftarrow}

\newcommand{\al}{\alpha}
\renewcommand{\phi}{\varphi}

\renewcommand{\th}{\theta}
\newcommand{\ga}{\gamma}
\newcommand{\gar}{\mathrel{\gamma}}
\newcommand{\thr}{\mathrel{\theta}}

\newcommand{\nthr}{\mathrel{\cancel{\theta}}}

\newcommand{\del}{\delta}
\newcommand{\delr}{\mathrel{\delta}}

\newcommand{\card}[1]{|#1|}

\newcommand{\st}{\mid}
\newcommand{\con}[1]{\mathrel{\langle #1 \rangle}}
\newcommand{\conr}[2]{\mathrel{\langle #1;\bar{#2} \rangle}}

\newcommand{\modt}{\mathrel{\equiv_t}}
\newcommand{\mods}{\mathrel{\equiv_s}}
\newcommand{\modkk}{\mathrel{\equiv_{2k}}}
\newcommand{\modku}{\mathrel{\equiv_{2k+1}}}
\newcommand{\modkd}{\mathrel{\equiv_{2k+2}}}

\newcommand{\br}{{\bar r}}
\newcommand{\bs}{{\bar s}}
\newcommand{\ext}{\mathrm{ext}}
\newcommand{\tira}[1]{\mathbf{#1}}
\newcommand{\displ}[1]{\Delta_{\bar #1}}
\newcommand{\alt}{\mathsf{f}}
\newcommand{\rr}{\mathrel{R}}
\renewcommand{\eta}{e}

\definecolor{lightblue}{RGB}{224,224,255}
\definecolor{lightred}{RGB}{255,224,224}
\definecolor{lightgreen}{RGB}{224,255,224}
\definecolor{lightyellow}{RGB}{255,255,224}
\definecolor{lightpurple}{RGB}{255,224,255}
\definecolor{darkerred}{RGB}{64,0,0}
\definecolor{darkred}{RGB}{128,0,0}
\definecolor{darkblue}{RGB}{0,0,128}
\definecolor{darkgreen}{RGB}{0,128,0}
\definecolor{darkpurple}{RGB}{128,0,128}

\makeatletter
\def\THICKhrulefill{\leavevmode \leaders \hrule height 5pt\hfill \kern \z@}
\makeatother

\newtheorem{theorem}{Theorem}
\newtheorem{thm}[theorem]{Theorem}
\newtheorem{lemma}[theorem]{Lemma}
\newtheorem{lem}[theorem]{Lemma}

\newtheorem{corollary}[theorem]{Corollary}
\newtheorem{cor}[theorem]{Corollary}

\newtheorem*{claim*}{Claim}
\theoremstyle{definition}

\newtheorem{defn}[theorem]{Definition}
\newtheorem*{notation}{Notation}
\newtheorem{example}[theorem]{Example}
\theoremstyle{remark}
\newtheorem*{ack}{Acknowledgments}
\newtheorem{rem}[theorem]{Remark}

\begin{document}
\title{The Lattice of Congruences of a Finite Line Frame}
\author{C.~Areces
  \and M.~Campercholi \and D.~Penazzi \and P.~S\'{a}nchez Terraf%
}
\maketitle


\maketitle

\begin{abstract}
Let $\mathbf{F}=\left\langle F,R\right\rangle $ be a finite Kripke
frame. A \emph{congruence} of $\mathbf{F}$ is a bisimulation of $\mathbf{F}$
that is also an equivalence relation on F. The set of all congruences
of $\mathbf{F}$ is a lattice under the inclusion ordering. In this
article we investigate this lattice in the case that $\mathbf{F}$
is a finite line frame. We give concrete descriptions of the join
and meet of two congruences with a nontrivial upper bound. Through
these descriptions we show that for every nontrivial congruence $\rho$,
the interval $[\mathrm{Id_{F},\rho]}$ embeds into the lattice of
divisors of a suitable positive integer. We also prove that any two
congruences with a nontrivial upper bound permute.
\end{abstract}
%

%
%
%
%
%

%
\section{Introduction}

Let $\mathbf{F}=\left\langle F,R\right\rangle $ be a finite Kripke
frame (i.e., directed graph). A \emph{congruence} of $\mathbf{F}$
is a bisimulation of $\mathbf{F}$ that is also an equivalence relation
on F. The set of all congruences of $\mathbf{F}$
is a lattice under the inclusion ordering. In this article we investigate
congruence lattices of finite \emph{line frames} (also called ``lines'' in
the sequel). 
To be precise, a \emph{line} is a finite frame $\mathbf{L}=\langle\{0,\ldots,n\},R\rangle$,
where $x\mathrel{R}y$ iff $|x-y|\leq1$. We are able to give concrete
descriptions of the join and meet of two congruences with a nontrivial
upper bound. Through these descriptions we show that for every nontrivial
congruence $\rho$, the interval $[\mathrm{Id_{L},\rho]}$ embeds
in a natural way into the lattice of divisors of a suitable positive
integer. We also prove that any two congruences with a nontrivial
upper bound permute; that is their join is the composition.

The congruences of a line have an appealing geometrical nature, since
a congruence can be thought of as a way of \emph{folding} the line.
It is also possible to give a geometrical representation of two congruences
on a diagram resembling the \emph{trajectory} of a particle traveling
inside a rectangle and bouncing along its sides. These representations
turn the problem of understanding the congruence lattice of a line
into a surprisingly elementary geometrical puzzle.

\subsection{Background}

Our interest in congruences of frames arises from our study of equationally
definable functions (called \emph{algebraic functions }\cite{algebraicFunc})
in modal algebras. Efforts to characterize this kind of functions
for other classes of algebraic structures (e.g., Boolean algebras,
distributive lattices, abelian groups, etc.) has been the focus of
\cite{algebraicFunc}, \cite{AlgebraicasLuka} and \cite{AlgebraicasQuasi}.
This line of research is motivated by the understanding that many
structural properties of algebras are tied to syntactical phenomena.
Mal'cev conditions \cite{BurrisSankappanavar1981} are a premier example
of this fact, and a variety of fundamental results in Universal Algebra
are of this nature. Thus, studying the term-functions of the algebras
in a given class is a powerful approach to understand structural properties.
Term-functions are the most basic \emph{definable functions}, and
oftentimes characterizations require more complex syntactical counterparts.
Algebraic functions are a natural generalization of term-functions;
they share many of the basic properties of term-functions and every
term-function is an algebraic function \cite{algebraicFunc}. An interesting
example of a result linking equationally definable functions to structural
properties can be found in \cite{QuaternaryMono}.

In \cite{algebraicFunc} it is shown that the study of algebraic functions
in the algebras of a variety $\mathcal{V}$ can be approached through
the investigation of the subclasses of $\mathcal{V}$ axiomatizable
by sentences of the form $\forall\bar{x}\exists_{=1}\bar{y}\,\alpha(\bar{x},\bar{y})$,
where $\alpha(\bar{x},\bar{y})$ is a conjunction of term-equalities.
We call these formulas \emph{EFD-sentences} (EFD stands for equational
function definition). In the case that $\mathcal{V}$ is a finitely
generated discriminator variety (in particular, when $\mathcal{V}$
is a finitely generated variety of modal algebras \cite{KrachtSemisimpleModal})
the EFD-axiomatizable subclasses of $\mathcal{V}$ are in correspondence
with the subclasses of simple members of $\mathcal{V}$ closed under
isomorphisms, fixed-point subalgebras and intersection of subalgebras
\cite{teoReferee}. Thus, we set out to characterize the classes $\mathcal{C}$
of finite simple modal algebras satisfying the above closure conditions.
This turned out to be a very challenging task. The key difficulty
comes from the fact that there is no obvious description of the atoms
of an intersection of two subalgebras of a modal algebra. Our tool
of choice to tackle this problem is the duality linking finite modal
algebras and Kripke frames \cite{libroKracht}. The lattice of subuniverses
of a finite modal algebra is dually isomorphic with the lattice of
congruences of its dual frame. So understanding intersection of subalgebras
amounts to the same thing as understanding the join of congruences
of frames. Finding a simple description for the join of two congruences
of an arbitrary finite frame appears to be a very hard (perhaps impossible)
task. In this article we focus on line frames, and already for this
special case we found it to be an interesting problem.

It is worth mentioning that we did not include here any of the (universal)
algebraic applications of the results, as we plan to address them
in a forthcoming article.

\paragraph{Outline}
In the next section we study the basic properties of congruences
of lines and provide an arithmetical characterization for them. We also introduce
their description as foldings; this representation proves to be very useful for getting insight
and simplifying proofs. In particular, it will lead to the 
characterization of the order among congruences.

In Section~\ref{sec:case-frequency-two} we address a special case of joins,
namely, when one of the congruences has equivalence classes with at 
most two elements. This case is treated separately because it
does not follow the general pattern that joins of congruences with bigger
classes enjoy. This pattern is captured with the help of trajectories, which we introduce
in Section~\ref{sec:trajectories}. 

Our main results are gathered in
Section~\ref{sec:main-results}. Section~\ref{sec:catalog-non-trivial-joins}
provides a complete catalog of the possible local configurations of
pairs of congruences having nontrivial joins. This catalog will help
in proving permutability in Section~\ref{sec:permutability}. Finally,
we obtain a  simple formula for calculating the quotient
of a line by a nontrivial join in Section~\ref{sec:computing-join}.

%
%
%
%
%

\section{Congruences of Lines}

\subsection{Congruences of a frame}

Let $\mathbf{F}=\left\langle F,R\right\rangle $ be a frame. Recall
that a \emph{bisimulation} of $\mathbf{F}$ is a binary relation $\theta\subseteq F\times F$
satisfying $\theta\circ R\subseteq R\circ\theta$ and $\theta^{-1}\circ R\subseteq R\circ\theta^{-1}$.
We call a bisimulation of $\mathbf{F}$ that is also an equivalence
relation on $F$ a \emph{congruence} of $\mathbf{F}$. We write $\Con\mathbf{F}$
to denote the set of congruences of $\mathbf{F}$. Since equivalence
relations are symmetric, $\theta$ is a congruence of $\mathbf{F}$
if and only if 
\[
\left(\forall x,x',y\in F\right)\ x'\thr x\mbox{ and }x\rr y\text{ imply there is }y'\text{ such that }x'\rr y'\mbox{ and }y'\thr y.
\]
Observe that when $R$ is symmetric, the congruences of $\mathbf{F}$
are the equivalence relations $\theta$ on F that permute with $R$,
i.e., such that $\theta\circ R=R\circ\theta$.

We say that an equivalence relation is \emph{trivial} if it has exactly
one equivalence class. The \emph{quotient} of $\mathbf{F}$ by $\theta$
is the frame $\mathbf{F}/\theta\doteq\left\langle F/\theta,R_{\theta}\right\rangle $
where $x/\theta\mathrel{R_{\theta}}y/\theta$ iff there are $x'\thr x$
and $y'\thr y$ such that $x'\rr y'$. For $x\in F$ we define $R\left[x\right]\doteq\left\{ y\in F\mid x\mathrel{R}y\right\} $
and $v\left(x\right)\doteq\left|R\left[x\right]\right|$. Here are
some basic properties of congruences and quotients.
\begin{lem}
\label{lem:basicas-congruencias}Let $\theta$ be a congruence of
the frame $\mathbf{F}=\left\langle F,R\right\rangle $.
\begin{enumerate}
\item $x\rr y$ implies $x/\theta\rr_{\theta}y/\theta$.
\item If there is a $R$-path from $x$ to $y$ in $\mathbf{F}$, then there
is one from $x/\theta$ to $y/\theta$ in $\mathbf{F}/\theta$. Thus
if $\mathbf{F}$ is connected so is $\mathbf{F}/\theta$.
\item $R_{\theta}\left[x/\theta\right]=R\left[x\right]/\theta$.
\item $v\left(x/\theta\right)\leq v\left(x\right)$.
\item If $\mathbf{F}$ is connected and there is $x\in F$ such that \textup{$R\left[x\right]\subseteq x/\theta$
then $\theta$ is trivial.}
\end{enumerate}
\end{lem}
In what follows we may omit the subscript $\theta$ in $R_{\theta}$
whenever there is no risk of confusion.

\subsection{Congruences of a line}

For an integer $n\geq0$ let $\mathbf{L}_{n}\doteq\left\langle \{0,\ldots,n\},R\right\rangle $
where $x\rr y$ iff $\left|x-y\right|\leq1$. Note that $x\rr x$
for all $x$. We call such a frame a \emph{line}. The remainder of
this section is devoted to characterizing the congruences of $\mathbf{L}_{n}$.
\begin{lem}
\label{lem:basicas-con-Ln}Let $\theta$ be a congruence of $\mathbf{L}_{n}$.
\begin{enumerate}
\item $\mathbf{L}_{n}/\theta$ is a line.
\item If $a\thr b$ and $a<b$, then there is $x\in\left[a,b\right)$ such
that $v\left(x/\theta\right)\leq2$.
\item Suppose $\theta$ is nontrivial, and let $k=\min\left\{ x\in\left[1,n\right]\mid v\left(x/\theta\right)=2\right\} $.
Then $\mathbf{L}_{n}/\theta$ is isomorphic with $\mathbf{L}_{k}$
and $0/\theta,\ldots,k/\theta$ are all the equivalence classes of
$\theta$.
\end{enumerate}
\end{lem}
\begin{proof}
1. Note that by Lemma \ref{lem:basicas-congruencias} $\mathbf{L}_{n}/\theta$
must be connected and $v\left(x/\theta\right)\leq3$ for all $x$.
Furthermore, $v\left(0/\theta\right)\leq2$ and thus the quotient
has to be a line.

\noindent 2. If $a=0$ there is nothing to prove since $v\left(0/\theta\right)\leq2$,
so we may suppose $a>0$. If $b=a+1$, then $R\left[a/\theta\right]=R\left[a\right]/\theta=\{a-1/\theta,a/\theta,a+1/\theta\}$,
but this is a contradiction since $a\thr a+1$ and $\theta$ is nontrivial.
If $b=a+2$, then $R\left[a+1/\theta\right]=R\left[a+1\right]/\theta=\{a/\theta,a+1/\theta,a+2/\theta\}$,
and it follows that $v\left(a+1/\theta\right)\leq2$. Assume $b-a>2$.
We proceed now by induction in $b$. The case $b=1$ is vacuously
true (since $a>0$). Suppose $b>1$ and observe that 
\[
b-1/\theta\in R\left[b/\theta\right]=R\left[a/\theta\right]=R\left[a\right]/\theta=\{a-1/\theta,a/\theta,a+1/\theta\}.
\]
If $b-1/\theta\in\{a/\theta,a+1/\theta\}$, we are done by inductive
hypothesis (note that by our assumptions $a+1<b-1$). So it only remains
to deal with case $b-1\thr a-1$. By inductive hypothesis there is
$x\in\left[a-1,b-1\right)$ such that $v\left(x/\theta\right)\leq2$.
If $x=a-1$ then $v\left(b-1/\theta\right)=v\left(a-1/\theta\right)\leq2$.
If on the other hand $x\neq a-1$, then $x\in\left[a,b\right)$.

\noindent 3. Observe that such a $k$ must always exist. In fact,
by 1 we know that $\mathbf{L}_{n}/\theta$ is a line, and since $\theta$
is not trivial this line has more than one point. One of the endpoints
is $0/\theta$ so there must be a nonzero element whose equivalence
class is the other endpoint. Also observe that as $\theta$ is not
trivial we have $v\left(x/\theta\right)>1$ for all $x$.

If there are $a,b\in\left[0,k\right]$ with $a\thr b$ and $a<b$,
then by 2 there is $x\in\left[a,b\right)$ such that $v\left(x/\theta\right)=2$,
in contradiction with the minimality of $k$. Thus $0/\theta,\ldots,k/\theta$
are pairwise distinct and $\left(0/\theta\right)\rr\left(1/\theta\right)\rr\cdots\rr\left(k/\theta\right)$.
Finally, as $\mathbf{L}_{n}/\theta$ is a line, there cannot be any
more equivalence classes.
\end{proof}
Given a congruence $\theta$ of $\mathbf{L}_{n}$ we define the \emph{step}
of $\theta$ to be the unique $k$ such that $\mathbf{L}_{n}/\theta$
is isomorphic with $\mathbf{L}_{k}$, i.e., the step of $\theta$
is $\left|L_{n}/\theta\right|-1$. Observe that 3 in the above lemma
shows how to compute the step of a nontrivial congruence. A \emph{rest}
of $\theta$ is a pair in $\theta$ of the form $\left\langle r,r+1\right\rangle $.
Here $r$ is the \emph{left part} and $r+1$ is the \emph{right part}
of the rest. We write \emph{$\theta$ has a rest at $r$} to express
that $\left\langle r,r+1\right\rangle $ is a rest of $\theta$. The
choice of terminology shall become clear once we introduce the geometrical
representation of congruences as foldings of the line.

Next we introduce a family of equivalence relations on the set of
integers. Given a (possibly empty) strictly increasing sequence of
non-negative integers $\bar{r}=\left\langle r_{1},\ldots,r_{m}\right\rangle $
let $\Delta_{\bar{r}}:\mathbb{Z}\rightarrow\mathbb{Z}$ be defined
by 
\[
\Delta_{\bar{r}}\left(x\right)\doteq\left|\left\{ i\mid r_{i}<x\right\} \right|.
\]
Observe that when $\bar{r}$ is the empty sequence then $\Delta_{\bar{r}}$
is constantly $0$. We often drop the subscript in $\Delta_{\bar{r}}$
when $\bar{r}$ is clear from the context.

For a positive integer $q$ let $\equiv_{q}$ denote the equivalence
relation modulo $q$ on the integers.
\begin{defn}
For an integer $k\geq1$ and a (possibly empty) strictly increasing
sequence of non-negative integers $\bar{r}$, let $\conr{k}{r}$ be
the binary relation defined on the set of integers by
\[
x\conr{k}{r}y\textnormal{ iff }\left[x-\Delta_{\bar{r}}\left(x\right)\equiv_{2k}y-\Delta_{\bar{r}}\left(y\right)\right]\textnormal{ or }\left[x-\Delta_{\bar{r}}\left(x\right)\equiv_{2k}-y+\Delta_{\bar{r}}\left(y\right)\right].
\]

\end{defn}
When $\bar{r}$ is the empty sequence we just write $\con{k}$. Note
that in this case we have $x\con{k}y$ iff $x\equiv_{2k}\pm y$.

To avoid tiresome repetitions, from this point onward $k$ shall always
denote a positive integer and $\bar{r}$ a strictly increasing sequence
of non-negative integers of length $\left|\bar{r}\right|$. We write
$a\in\bar{r}$ to signify that $a$ equals some entry of $\bar{r}$.

It is straightforward to check that $\conr{k}{r}$ is an equivalence
relation on $\mathbb{Z}$. Thus, so is its restriction to $\left\{ 0,\ldots,n\right\} $,
which we shall also denote by $\conr{k}{r}$. To get a better understanding
of these relations it is convenient to introduce a graphical representation
for them. To make a picture of $\conr{k}{r}$ on $\mathbf{L}_{n}$
we draw $\mathbf{L}_{n}$ `folded' in such a way that two points are
at the same height if and only if they are $\conr{k}{r}$-related.
Our convention is to place $0$ (and thus every point in the same
block) at the bottom level. Figure \ref{fig:ej-plegados} shows two
examples.

\begin{figure}
\begin{centering}
\subfloat{\begin{tikzpicture}[line cap=round,line join=round,>=triangle 45,x=0.7cm,y=0.7cm]
\draw   (10,10)-- (11,11)-- (12,10)-- (13,11)-- (14,10)-- (15,11);
\fill [color=black] (10,10) circle (1.6pt);
\draw[color=black] (10.04,10.4) node {$0$};
\fill [color=black] (11,11) circle (1.6pt);
\draw[color=black] (11.01,11.4) node {$1$};
\fill [color=black] (12,10) circle (1.6pt);
\draw[color=black] (12.01,10.4) node {$2$};
\fill [color=black] (13,11) circle (1.6pt);
\draw[color=black] (13,11.4) node {$3$};
\fill [color=black] (14,10) circle (1.6pt);
\draw[color=black] (14,10.4) node {$4$};
\fill [color=black] (15,11) circle (1.6pt);
\draw[color=black] (14.99,11.4) node {$5$};
\end{tikzpicture}}\qquad{}\subfloat{\begin{tikzpicture}[line cap=round,line join=round,>=triangle 45,x=0.7cm,y=0.7cm]
\draw   (10,10)-- (11,11)-- (12,12)-- (13,12)-- (14,11)-- (15,10);
          \fill [color=black] (10,10) circle (1.6pt);           
\draw[color=black] (10.04,10.5) node {$0$};
          \fill [color=black] (11,11) circle (1.6pt);
          \draw[color=black] (11.01,11.5) node {$1$};
          \fill [color=black] (12,12) circle (1.6pt);
          \draw[color=black] (12.01,12.4) node {$2$};
          \fill [color=black] (13,12) circle (1.6pt);
          \draw[color=black] (13,12.4) node {$3$};
          \fill [color=black] (14,11) circle (1.6pt);
          \draw[color=black] (14,11.5) node {$4$};
          \fill [color=black] (15,10) circle (1.6pt);
          \draw[color=black] (14.99,10.5) node {$5$};
        \end{tikzpicture}}
\par\end{centering}

\protect\caption{Relations $\left\langle 1\right\rangle $ and $\left\langle 2;2\right\rangle $
on $\mathbf{L}_{5}$ }
\label{fig:ej-plegados}
\end{figure}
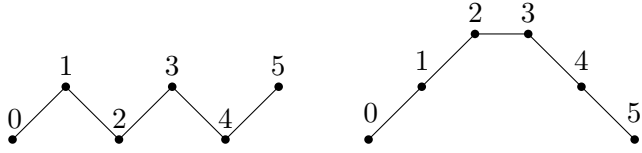

Next we show that every congruence of $\mathbf{L}_{n}$ is of this
type.
\begin{lem}
\label{lem:toda_con_es_un_par}Let $\theta$ be a nontrivial congruence
of $\mathbf{L}_{n}$ with step $k$. Suppose\linebreak{}
$\mbox{\ensuremath{\left\{ x\mid x\thr x+1\right\} =\left\{ r_{1},\ldots,r_{m}\right\} }}$,
where $m\geq0$ and $r_{1}<\ldots<r_{m}$ . Then $\theta={\conr{k}{r}}$.\end{lem}
\begin{proof}
Let us write $\Delta$ for $\Delta_{\bar{r}}$. Our first observation
is that point 3 of Lemma \ref{lem:basicas-con-Ln} implies $r_{1}\geq k$,
and hence 
\begin{equation}
\Delta\left(j\right)=0\text{ for all }j\in\left\{ 0,\ldots,k\right\} .\label{eq:delta_es_0}
\end{equation}
We prove by induction on $x$ that $\theta$ and $\conr{k}{r}$ agree
on the interval $\left[0,x\right]$. This is easily seen to be true
for $x\leq k$ using 3 of Lemma \ref{lem:basicas-con-Ln}, so assume
$x>k$. Note that $x-1\thr x$ iff $x-1\conr{k}{r}x$, thus we may
suppose 
\begin{equation}
x-1/\theta\neq x/\theta,\label{eq:x-1_norel_x}
\end{equation}
and so 
\begin{equation}
\Delta\left(x-1\right)=\Delta\left(x\right).\label{eq:deltas_iguales}
\end{equation}
Let $j\in\left[0,k\right]$ be such that $x-1\thr j$. By inductive
hypothesis we have

\[
x-1-\Delta\left(x-1\right)\equiv_{2k}j\text{ or }x-1-\Delta\left(x-1\right)\equiv_{2k}-j.
\]
There are three cases to consider.\medskip{}

\noindent Case $j=0$. We have $x-1-\Delta\left(x-1\right)\equiv_{2k}0$,
which by (\ref{eq:deltas_iguales}) implies $x-\Delta\left(x\right)\equiv_{2k}1$,
and this in combination with (\ref{eq:delta_es_0}) yields $x\conr{k}{r}1$.
Also, 
\[
x/\theta\in R\left[x-1/\theta\right]=R\left[0/\theta\right]=\left\{ 0/\theta,1/\theta\right\} ,
\]
 so (\ref{eq:x-1_norel_x}) says that $x\thr1$.\medskip{}

\noindent Case $j=k$. Note that 
\[
x/\theta\in R\left[x-1/\theta\right]=R\left[k/\theta\right]=\left\{ k-1/\theta,k/\theta\right\} ,
\]
and so by (\ref{eq:x-1_norel_x}) we have $x\thr k-1$. Also, $x-1-\Delta\left(x-1\right)\equiv_{2k}-k$
and (\ref{eq:deltas_iguales}) imply $x-\Delta\left(x\right)\equiv_{2k}-\left(k-1\right)$,
which in view of (\ref{eq:delta_es_0}) says that $x\conr{k}{r}k-1$.\medskip{}

\noindent Case $1<j<k$. Observe that 
\[
\left\{ x-2/\theta,x-1/\theta,x/\theta\right\} =R\left[x-1/\theta\right]=R\left[j/\theta\right]=\left\{ j-1/\theta,j/\theta,j+1/\theta\right\} ,
\]
where the three elements in the right-hand side set are different.
So, 
\begin{equation}
\left\{ x-2/\theta,x/\theta\right\} =\left\{ j-1/\theta,j+1/\theta\right\} ,\label{eq:2a2}
\end{equation}
and since $x-2/\theta\neq x-1/\theta$
\[
\Delta\left(x-1\right)=\Delta\left(x-2\right).
\]
Suppose first that 
\[
x-1-\Delta\left(x-1\right)\equiv_{2k}j.
\]
Then,
\begin{align*}
x-2-\Delta\left(x-2\right) & \equiv_{2k}j-1\\
x-\Delta\left(x\right) & \equiv_{2k}j+1,
\end{align*}
or equivalently 
\begin{align*}
x-2 & \conr{k}{r}j-1\\
x & \conr{k}{r}j+1.
\end{align*}
So, by inductive hypothesis $x-2\thr j-1$, and from (\ref{eq:2a2})
we obtain $x\thr j+1$.

The case $x-1-\Delta\left(x-1\right)\equiv_{2k}-j$ is handled in
the same way.\end{proof}
\begin{cor}
A congruence of $\mathbf{L}_{n}$ is determined by its step and rests.
\end{cor}
The next lemma tells us which of the relations $\conr{k}{r}$ are
congruences of $\mathbf{L}_{n}$. Note that when considering $\conr{k}{r}$
on $\left[0,n\right]$ we can safely assume $\bar{r}\subseteq\left[0,n-1\right]$,
due to the fact that adding or removing entries of $\bar{r}$ greater
than $n-1$ does not change the values of $\Delta_{\bar{r}}$ on $\left[0,n\right]$.
\begin{lem}
\label{lem:carac-pares-que-son-cong}Let $n\geq3$, and suppose $r_{1},\ldots,r_{m}=\bar{r}\subseteq\left[0,n-1\right]$
and $k\leq\frac{n}{2}$. Then $\conr{k}{r}$ is a nontrivial congruence
of $\mathbf{L}_{n}$ if and only if the following conditions hold:
\begin{enumerate}
\item $r_{1}\neq0$, $r_{m}\neq n-1$ and $r_{i+1}-r_{i}>2$, for all $i\in\left\{ 1,\ldots,\left|\bar{r}\right|-1\right\} $,
\item $k$ divides $r_{i}-i+1$,  for all $i\in\left\{
  1,\ldots,\left|\bar{r}\right|\right\} $, and 
\item $k$ divides $n-\left|\bar{r}\right|$.
\end{enumerate}
\end{lem}
\begin{proof}
Let us write $\theta$ for $\conr{k}{r}$. The lemma for $m=0$ is
an easy exercise and we leave it to the reader. Assume $m\geq1$.
We prove first that if conditions (1-3) hold, then $\theta$ is a
congruence of $\mathbf{L}_{n}$. As we know that $\theta$ is an equivalence
relation we only need to check that given $a,a',b\in\left[0,n\right]$
such that $a'\thr a\rr b$, there is $b'$ satisfying $a'\rr b'\thr b$.

When $a=a'$ there is nothing to prove, so assume $a\neq a'$. Also,
if $a\thr b$ we can take $b'=a'$, so we suppose 
\begin{equation}
\left\langle a,b\right\rangle \notin\theta.\label{eq:ab_not_in_theta}
\end{equation}
There are four cases to consider.

\noindent Case $a-\Delta\left(a\right)\equiv_{2k}a'-\Delta\left(a'\right)$
and $b=a+1$. 

Observe that (\ref{eq:ab_not_in_theta}) implies $a\notin\bar{r}$.
We consider first the sub-case $a'\in\bar{r}$. By 1 we know that $a'\neq0$
and $a'-1\notin\bar{r}$, so $\Delta\left(a'-1\right)=\Delta\left(a'\right)$.
By 2 we have that $k$ divides $a'-\Delta\left(a'\right)$, so $-a'+\Delta\left(a'\right)\equiv_{2k}a'-\Delta\left(a'\right)$.
Hence, 
\[
-\left(a'-1-\Delta\left(a'-1\right)\right)\equiv_{2k}a+1-\Delta\left(a+1\right).
\]
That is, $a'-1\thr a+1$, and we can take $b'=a'-1$. Observe that
the same proof also works for the case $a'=n$ using 3 instead of
2. So the only sub-case left to address is $a'\notin\bar{r}$ and $a'<n$.
Here it is easy to see that $b'=a'+1$ works.\medskip{}

\noindent Case $a-\Delta\left(a\right)\equiv_{2k}a'-\Delta\left(a'\right)$
and $b=a-1$.

From (\ref{eq:ab_not_in_theta}) we obtain $a-1\notin\bar{r}$. If
$a'=0$ or $a'-1\in\bar{r}$ a reasoning analogous to the one used
in the sub case $a'\in\bar{r}$ above shows that $a'+1\thr a-1$,
and so we can take $b'=a'+1$. If $a'\neq0$ and $a'-1\notin\bar{r}$,
then it is straightforward to check that $b'=a'-1$ does the job.\medskip{}

\noindent The two remaining cases are left to the reader.

Next suppose $\theta$ is a nontrivial congruence of $\mathbf{L}_{n}$.
We prove that (1-3) hold. First observe that if 1 does not hold then
there is $x\in[0,n]$ such that $R\left[x\right]\subseteq x/\theta$,
and thus $\theta$ would be trivial by 5 of Lemma \ref{lem:basicas-congruencias}.
To establish 2 and 3 we need to show first that $k\leq r_{1}$. For
the sake of contradiction suppose $r_{1}<k$. By Lemma \ref{lem:basicas-con-Ln},
$v(r_{1})=2$. From the definition of $\conr{k}{r}$ it is easily
checked that $r_{1}$ is the minimal element with such property. So
by Lemma \ref{lem:basicas-con-Ln} we have that $r_{1}$ is the step
of $\theta$ and $0/\theta,\ldots,r_{1}/\theta$ are all the $\theta$-blocks.
In particular $r_{1}+2\thr j$, for some $j\in[0,r_{1}]$. So 
\[
r_{1}+1=r_{1}+2-\Delta\left(r_{1}+2\right)\equiv_{2k}\pm j.
\]
 That is, $2k$ divides either $r_{1}+1-j$ or $r_{1}+1+j$. Both
cases easily yield a contradiction.

Now, if $k\leq r_{1}$ then $k$ is the step of $\theta$, and by
Lemma \ref{lem:basicas-con-Ln} we know that $0/\theta,\ldots,k/\theta$
are all the equivalence classes of $\theta$. Observe that $v\left(r_{i}/\theta\right)=2$,
and thus $r_{i}/\theta\in\left\{ 0/\theta,k/\theta\right\} $, for
all $i\in\left\{ 1,\ldots,m\right\} $. Analogously $n/\theta\in\left\{ 0/\theta,k/\theta\right\} $.
From here 2 and 3 are easily obtained.
\end{proof}
Combining Lemmas \ref{lem:toda_con_es_un_par} and \ref{lem:carac-pares-que-son-cong}
we obtain a complete description of the congruences of a line.
\begin{thm}
\label{th:charact-congruences}The congruences of $\mathbf{L}_{n}$
are the diagonal, $L_{n}\times L$$_{n}$, and all the relations $\conr{k}{r}$
such that:
\begin{itemize}
\item $k\leq\frac{n}{2}$,
\item $\bar{r}\subseteq\left[k,n-k\right]$,
\item $k$ divides $r_{i}-i+1$, for all $i\in\left\{
  1,\ldots,\left|\bar{r}\right|\right\} $, and 
\item $k$ divides $n-\left|\bar{r}\right|$.
\end{itemize}
\end{thm}
We can rewrite the conditions in the theorem above to characterize
the foldings of $\mathbf{L}_{n}$ that represent nontrivial congruences.
\begin{itemize}
\item The vertical segments have all the same length.
\item The horizontal segments (rests) of the folding must be at the top
and bottom levels.
\item Consecutive rests are not allowed.
\item Rests starting at $0$ or ending at $n$ are not allowed.
\item $n$ must be at the top or bottom level.
\end{itemize}

\subsection{Basic properties of congruences}

It turns out that for two congruences to have a nontrivial join both
must have a highly regular and compatible structure. Thus a common
theme throughout the paper is that most results apply to this case.
In the current section we start to pin down the aforementioned regularity,
and characterize the meet of two congruences with nontrivial join.
The characterization of the join takes quite a bit more of work, and
is handled in the subsequent sections.

For the remainder of the current section we fix a natural number $n$,
and write $\theta$ and $\delta$ to denote congruences of $\mathbf{L}_{n}$. 

We start out with a list of equivalent conditions for a congruence
to be trivial.
\begin{lem}
\label{lem:carac-trivial}Suppose $n\geq2$. The following are equivalent:
\begin{enumerate}
\item $\theta$ is trivial.
\item There is $x$ such that $x-1\thr x\thr x+1$.
\item $0\thr1$.
\item $n-1\thr n$.
\end{enumerate}
\end{lem}
\begin{proof}
Immediate by 5 of Lemma \ref{lem:basicas-congruencias}.
\end{proof}
We say that $e\in[0,n]$ is an \emph{extreme} of $\theta$ if $e/\theta$
is one of the endpoints of $\mathbf{L}_{n}/\theta$. We denote the
set of all extremes of $\theta$ by $\ext(\theta)$. Observe that
if $k$ is the step of $\theta$ we have 
\[
\ext(\theta)=0/\theta\cup k/\theta.
\]
Note that $0$ and $n$ are always extremes, and if $r$ is part of
a rest then $r\in\ext\left(\theta\right)$. Also note that the step
of $\theta$ equals its first positive extreme.

Suppose $e\notin\{0,n\}$ is an extreme of $\theta$ that is not part
of a rest. If we take two points, one on each side of $e$, that are
at the same distance $d$ from $e$ they will be $\theta$-related
as long as $d$ is small enough. For instance this is always true
if $d$ is not greater than the step of $\theta$. In fact, it is
evident from looking at the folding associated to $\theta$ that as
we start moving away from $e$, equidistant points will be related
as long as we don't hit a rest on one side that is missing from the
other. So, the symmetry breaks for the first time when we encounter
an asymmetrical rest or we fall off the line on one side and not the
other. See Figure \ref{fig:simetria}.

When there is a rest at $e$ the situation is the same, only that
the symmetry is with respect to $e+\frac{1}{2}$. These facts are
summarized in a precise way below.

\begin{figure}
\begin{center}
\begin{tikzpicture}[line cap=round,line join=round,>=triangle 45,x=0.7cm,y=0.7cm] \draw (0,2)-- (1,4)-- (2,4)-- (4,0)-- (6,4)-- (8,0)-- (10,4)-- (11.5,1);         \fill [color=black] (1,4) circle (1.6pt);         \draw[color=black] (0.45,4.4) node {$e-x$};         \fill [color=black] (2,4) circle (1.6pt);         \draw[color=black] (2.87,4.42) node {$e-x+1$};         \fill [color=black] (6,4) circle (1.6pt);         \draw[color=black] (5.98,4.4) node {$e$};         \fill [color=black] (10,4) circle (1.6pt);         \draw[color=black] (9.97,4.41) node {$e+x-1$};         \fill [color=black] (3.52,1) circle (1.6pt);         \draw[color=black] (2.67,1.07) node {$e-t$};         \fill [color=black] (10.49,3) circle (1.6pt);         \draw[color=black] (11.2,3.09) node {$e+x$};         \fill [color=black] (8.51,1) circle (1.6pt);         \draw[color=black] (9.3,1.01) node {$e+t$};
\end{tikzpicture}
\end{center}

\protect\caption{Moving away from an extreme}
\label{fig:simetria}
\end{figure}
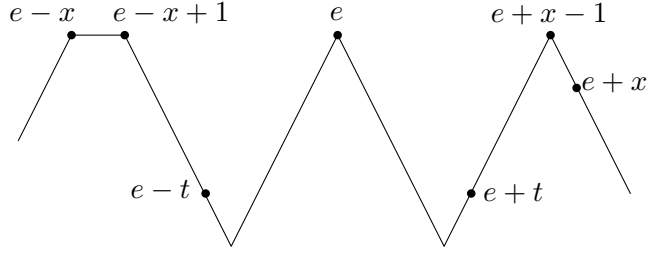

\begin{lem}
\label{lem:equidistantes-a-ext}Let $e\notin\{0,n\}$ be an extreme
of $\theta$.
\begin{enumerate}
\item Suppose $e$ is not part of a rest. Assume 
\[
\left\{ t\in\omega\mid e-t\mbox{ and }e+t\mbox{ are in
  \ensuremath{[0,n]} and are not }\theta\mbox{-related}\right\} \neq\emptyset,
\]
and let $x$ be the least member of this set. Then $e+x-1$ is an
extreme of $\theta$, and either

\begin{enumerate}
\item there is a rest at $e-x$ and $e+x-1$ is not part of a rest or 
\item there is a rest at $e+x-1$ and $e-x$ is not part of a rest.
\end{enumerate}
\item Suppose there is a rest at $e$. Assume 
\[
\left\{ t\in\omega\mid e-t\mbox{ and }e+t+1\mbox{ are in
  \ensuremath{[0,n]} and are not }\theta\mbox{-related}\right\} \neq\emptyset,
\]
and let $x$ be the least member of this set. Then $e+x$ is an extreme
of $\theta$, and either

\begin{enumerate}
\item there is a rest at $e-x$ and $e+x$ is not part of a rest or 
\item there is a rest at $e+x$ and $e-x$ is not part of a rest.
\end{enumerate}
\end{enumerate}
\end{lem}
Next we show that when the join of two congruences is not trivial
they behave exactly the same at common extremes.
\begin{lem}
\label{lem:coinciden-en-etas}Let $\theta$ and $\delta$ be congruences
of $\mathbf{L}_{n}$ such that $\theta\vee\delta$ is not trivial,
and let $e\in\ext(\theta)\cap\ext(\delta)$. Then $e$ is the left
(right) part of a rest of $\theta$ if and only if $e$ is the left
(right) part of a rest of $\delta$.\end{lem}
\begin{proof}
Suppose $e$ is the left part of a rest of $\theta$ and is not part
of a rest of $\delta$. Then $e\thr e+1$ and $e-1\delr e+1$, so
\[
e-1\mathrel{\left(\theta\vee\delta\right)}e\mathrel{\left(\theta\vee\delta\right)}e+1.
\]
 Thus $\theta\vee\delta$ is trivial by Lemma \ref{lem:carac-trivial}.
The remaining cases are just as easy.
\end{proof}
When applied to a pair of congruences with nontrivial join, the two
lemmas above show that each of these congruences is determined by
its behavior between any two consecutive common extremes that do not
constitute a rest.
\begin{lem}
\label{lem:periodicas}Assume $\theta\vee\delta$ is not trivial,
and let $e\in\ext(\theta)\cap\ext(\delta)$ such that $e\notin\{0,n\}$.
\begin{enumerate}
\item Suppose $e$ is not part of a rest of $\theta$, and let $d$ be the
distance from $e$ to the next greater element of $\ext(\theta)\cap\ext(\delta)$.
Then $e-d$ is the greatest element of $\ext(\theta)\cap\ext(\delta)$
before $e$ and 
\begin{align*}
e-t & \thr e+t\\
e-t & \delr e+t
\end{align*}
for $t=0,\ldots,d$. 
\item Suppose $\theta$ has a rest at $e$, and let $d$ be the distance
from $e+1$ to the next greater element of $\ext(\theta)\cap\ext(\delta)$.
Then $e-d$ is the greatest element of $\ext(\theta)\cap\ext(\delta)$
before $e$ and 
\begin{align*}
e-t & \thr e+t+1\\
e-t & \delr e+t+1
\end{align*}
for $t=0,\ldots,d$. 
\end{enumerate}
\end{lem}
\begin{proof}
1. For the sake of contradiction suppose there is $t\in[1,d]$ such
that $e-t$ and $e+t$ are not $\left(\theta\cap\delta\right)$-related.
Let $x$ be the smallest such $t$. Note that $e\geq x$, because
otherwise $0=e-e\mathrel{\left(\theta\cap\delta\right)}e+e$, and
as $0\in\ext(\theta)\cap\ext(\delta)$ it would follow that $2e$
is a common extreme satisfying $e<2e<e+d$.

So $e-x\geq0$ and $\left\langle e-x,e+x\right\rangle $ is either
not in $\theta$ or not $\delta$. We may suppose without loss that
$\left\langle e-x,e+x\right\rangle $ is not in $\theta$ (the other
case is symmetrical since by Lemma \ref{lem:coinciden-en-etas} we
have that $e$ is not part of a rest of $\delta$). Now by 1 of Lemma
\ref{lem:equidistantes-a-ext} we know that $e+x-1$ is an extreme
of $\theta$ and either (a) or (b) of 1 in that lemma hold. Assume
(a) holds, i.e., there is a rest of $\theta$ at $e-x$ and $e+x-1$
is not part of a rest of $\theta$. This says in particular that 
\[
e+x-2\thr e+x.
\]
We argue that $\left\langle e-x,e+x\right\rangle $ is not in $\delta$.
Suppose this is not the case, then 
\[
e+x\delr e-x\thr e-x+1\thr e+x-1,
\]
which in combination with the previous display implies that $e+x-2$,
$e+x-1$ and $e+x$ are all $\left(\theta\vee\delta\right)$-related,
making $\theta\vee\delta$ trivial by Lemma \ref{lem:carac-trivial}.
Thus $e-x$ and $e+x$ cannot be $\delta$-related. So, applying 1
of Lemma \ref{lem:equidistantes-a-ext} to $\delta$ it follows in
particular that $e+x-1$ is an extreme of $\delta$, and hence 
\[
e+x-1\in\ext(\theta)\cap\ext(\delta).
\]
 Note that as $\theta$ and $\delta$ are not trivial and $e$ is
not part of a rest of either congruence we have $e-1\mathrel{\left(\theta\cap\delta\right)}e+1$.
Thus $x>1$ and so 
\[
e<e+x-1<e+d.
\]
But then $e+x-1$ cannot be a common extreme by our choice of $d$.
A contradiction.

A similar argument works in the case that point (b) of 1 in Lemma
\ref{lem:equidistantes-a-ext} holds.

\noindent 2. This proof is the same as for 1 adjusting the formulas
for the rest at $e$ and applying 2 of Lemma \ref{lem:equidistantes-a-ext}
instead of 1.
\end{proof}
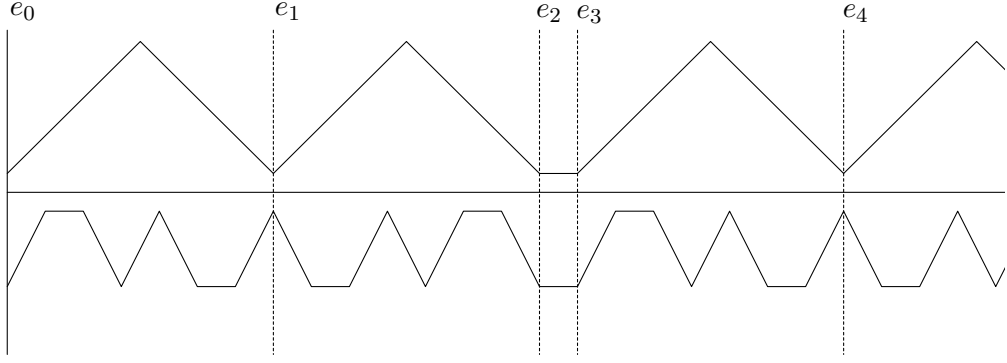
\begin{figure}
\begin{center}\begin{tikzpicture}[line cap=round,line join=round,>=triangle     45,x=0.25cm,y=0.25cm]       \clip(1.63,-4.23) rectangle   (54.47,15.47);     \draw [domain=2.0:59.14318456515393] plot(\x,{(--290-0*\x)/58}); \draw [dash pattern=on 1pt off 1pt] (16,13.6)-- (16,-3.6); \draw [dash pattern=on 1pt off 1pt] (30,13.6)-- (30,-3.6); \draw (29.28,15.49) node[anchor=north west] {$e_2$}; \draw (31.41,15.44) node[anchor=north west] {$e_3$}; \draw (45.41,15.49) node[anchor=north west] {$e_4$}; \draw (1.6,15.6) node[anchor=north west] {$e_0$}; \draw (2,-3.6)-- (2,13.6); \draw (2,6)-- (9,13)-- (16,6)-- (23,13)-- (30,6)-- (32,6)-- (39,13)-- (46,6)-- (53,13)-- (60,6); \draw (2,0)-- (4,4)-- (6,4)-- (8,0)-- (10,4)-- (12,0)-- (14,0)-- (16,4)-- (16,4)-- (18,0)-- (20,0)-- (22,4)-- (24,0)-- (26,4)-- (28,4)-- (30,0)-- (32,0)-- (34,4)-- (36,4)-- (38,0)-- (40,4)-- (42,0)-- (44,0)-- (46,4)-- (48,0)-- (50,0)-- (52,4)-- (54,0)-- (56,4)-- (58,4); \draw [dash pattern=on 1pt off 1pt] (32,13.6)-- (32,-3.6); \draw [dash pattern=on 1pt off 1pt] (46,13.6)-- (46,-3.6); \draw (15.54,15.6) node[anchor=north west] {$e_1$}; \end{tikzpicture}
\end{center}\protect\caption{Periodicity}
\label{fig:periodicas}
\end{figure}
Let us take closer look at what have just proved.
\begin{rem}
\label{rem:periodicidad}Take an enumeration $0=e_{0}<e_{1}<\cdots<e_{N}=n$
of $\ext(\theta)\cap\ext(\delta)$. If we draw the folding that represents
$\theta$, Lemma \ref{lem:periodicas} says that the part from $e_{1}$
to $e_{2}$ is the reflection of the part from $0$ to $e_{1}$ (assuming
there is no rest at $e_{1}$). Now the part from $e_{2}$ to $e_{3}$
is a reflection of the previous part and so on. Thus $\theta$ (and
$\delta$ too of course) is determined (up to rests at common extremes)
by its restriction to $\left[0,e_{1}\right]$. Figure \ref{fig:periodicas}
illustrates this fact.
\end{rem}

\subsection{The meet of two congruences}

Lemma \ref{lem:periodicas} allows us to characterize the meet of
two congruences with nontrivial join. But first we need to point out
a few easy facts.
\begin{lem}
\label{lem:aux-meet}Let $\theta$, $\delta$ and $\rho$ be congruences
of $\mathbf{L}_{n}$.
\begin{enumerate}
\item If $\rho\subseteq\theta$, then $\ext(\rho)\subseteq\ext(\theta)$.
\item If $\rho\varsubsetneq\theta$, then step of $\rho$ is strictly greater than the step of $\theta$.
\item If $\rho\subseteq\theta\cap\delta\neq L\times L$, then the step
  of $\rho$ is greater than or equal to 
  \[e_{1}\doteq\min\left\{ e\in\ext(\theta)\cap\ext(\delta)\mid
  e>0\right\}.
  \]
\item If $n=\min\left\{ e\in\ext(\theta)\cap\ext(\delta)\mid e>0\right\} $
and $\theta\cap\delta\neq L\times L$ then $\theta\wedge\delta=\mathrm{Id}$.
\end{enumerate}
\end{lem}
\begin{proof}
1. Let $e\in\ext(\rho)$. If $e\in\left\{ 0,n\right\} $ we are done,
so assume $0<e<n$. If $e$ is part of a rest of $\rho$, then $r\mathrel{\rho}r+1$
or $r\mathrel{\rho}r-1$. So $r$ must also be part of a rest of $\theta$
and thus $r\in\ext(\theta)$. If $e$ is not part of a rest of $\rho$
then $e-1\mathrel{\rho}e+1$, so $e-1\mathrel{\theta}e+1$ and $e$
must be an extreme of $\theta$.

\noindent 2. Recall that the step of a congruence is the size of its
quotient minus one.

\noindent 3. Let $q$ be the step of $\rho$. Since $q$ is an extreme
of $\rho$, by 1 we know that $q\in\ext(\theta)\cap\ext(\delta)$,
and as $\rho\neq L\times L$ we have $q>0$.

\noindent 4. By 3 any lower bound must have step at least $n$, and
thus the only lower bound is $\mathrm{Id}$.\end{proof}
\begin{thm}
\label{thm:meet}Let $\theta={\conr{k}{r}}$ and $\delta={\conr{l}{s}}$
be such that their join is not trivial. Then,
\[
\theta\wedge\delta=\left\langle e_{1};\bar{r}\cap\bar{s}\right\rangle ,
\]
where $e_{1}$ is the first positive element in $\ext(\theta)\cap\ext(\delta)$.\end{thm}
\begin{proof}
Let $0=e_{0}<e_{1}<\cdots<e_{N}=n$ be the common extremes of $\theta$
and $\delta$. Our first step is to prove that:\medskip{}

\noindent (i) if $e_{j}\notin\bar{r}\cap\bar{s}$ then, $e_{j+1}-e_{j}=e_{1}$,
for $j=0,\ldots,N-1$.\medskip{}

\noindent We prove it by induction in $j$. When $j=0$ it is certainly
true. Suppose $j>0;$ there are two cases. Assume first that $e_{j-1}\notin\bar{r}\cap\bar{s}$.
Then $e_{j}$ is not part of a rest of $\theta$, so by 1 of Lemma
\ref{lem:periodicas} we have 
\[
e_{j+1}-e_{j}=e_{j}-e_{j-1},
\]
and $e_{j}-e_{j-1}=e_{1}$ by inductive hypothesis. Suppose next that
$e_{j-1}\in\bar{r}\cap\bar{s}$. Then, by 2 of Lemma \ref{lem:periodicas}
it follows that 
\[
e_{j+1}-e_{j}=e_{j-1}-e_{j-2},
\]
and the right-hand side equals $e_{1}$ by inductive hypothesis.

An easy consequence of (i) is:\medskip{}

\noindent (ii) $e_{1}$ divides $e_{j}-\Delta_{\bar{r}\cap\bar{s}}\left(e_{j}\right)$,
for $j=0,\ldots,N$.\medskip{}

Let us write $\gamma$ for $\left\langle e_{1};\bar{r}\cap\bar{s}\right\rangle $.
Now we check that:\medskip{}

\noindent (iii) $\gamma$ is a congruence of $\mathbf{L}_{n}$.\medskip{}

\noindent In fact, the first two conditions in Theorem \ref{th:charact-congruences}
obviously hold, and the last two follow from (ii).

\noindent Our next step is to prove that $\gamma$ is a lower bound
of $\theta$ and $\delta$.\medskip{}

\noindent (iv) Let $x\in\left[0,n\right]$ and $y\in\left[0,e_{1}\right]$.
If $x\mathrel{\gamma}y$, then $x\mathrel{\left(\theta\cap\delta\right)}y$.\medskip{}

\noindent We proceed by induction in $x$. If $x\leq e_{1}$ it clearly
holds as in this case $x=y$. Suppose $x>e_{1}$, and let $e_{j},e_{j+1}$
be consecutive common extremes such that $e_{j}<x\leq e_{j+1}$. We
have three cases to consider. Assume first that neither $e_{j-1}$
nor $e_{j}$ are in $\bar{r}\cap\bar{s}$. Then it is easy to see
from the definition of $\gamma$ that 
\[
e_{j}-t\mathrel{\gamma}e_{j}+t\mbox{, for }t=0,\ldots,e_{j+1}-e_{j}.
\]
Also, by 1 of Lemma \ref{lem:periodicas}, the same holds for $\theta\cap\delta$,
that is 
\[
e_{j}-t\mathrel{\left(\theta\cap\delta\right)}e_{j}+t\mbox{, for }t=0,\ldots,e_{j+1}-e_{j}.
\]
In particular for $t=x-e_{j}$ we have 
\[
2e_{j}-x\mathrel{\left(\gamma\cap\theta\cap\delta\right)}x,
\]
and so by inductive hypothesis we know that $2e_{j}-x\mathrel{\left(\theta\cap\delta\right)}y$.
Thus $x\mathrel{\left(\theta\cap\delta\right)}y$.

In the case that $e_{j-1}\in\bar{r}\cap\bar{s}$ and $e_{j}\notin\bar{r}\cap\bar{s}$
the proof above works with minor adjustments. Finally, suppose $e_{j}\in\bar{r}\cap\bar{s}$.
Then $x=e_{j+1}$, and 
\[
e_{j}\mathrel{\left(\gamma\cap\theta\cap\delta\right)}e_{j+1}.
\]
Our inductive hypothesis yields $e_{j}\mathrel{\left(\theta\cap\delta\right)}y$,
and hence $x\mathrel{\left(\theta\cap\delta\right)}y$.

Observe that $0/\gamma,\ldots,e_{1}/\gamma$ are all the $\gamma$-blocks,
and thus (iv) immediately produces:\medskip{}

\noindent (v) $\gamma\subseteq\theta\cap\delta$.\medskip{}

It only remains to see that there cannot be a lower bound greater
than $\gamma$. Fix a congruence $\rho\subseteq\theta\cap\delta$,
and let $q$ be the step of $\rho$. By 3 of Lemma \ref{lem:aux-meet}
we have that $q\geq e_{1}$, and by 2 of the same lemma $\rho$ cannot
strictly contain $\gamma$. So we proved that $\gamma$ is a maximal
lower bound, but as the congruences of $\mathbf{L}_{n}$ form a lattice
$\gamma$ must be the meet.
\end{proof}
The \emph{frequency} of a congruence $\theta={\conr{k}{r}}$ of $\mathbf{L}_{n}$
is 
\[
\mathsf{f}_{\theta}\doteq\frac{n-\left|\bar{r}\right|}{k}.
\]
When thinking of $\theta$ as a folding, $\mathsf{f}_{\theta}$ is
the number of upward and downward slopes. Thus, another way to obtain
the frequency of $\theta$ is through counting extremes; this yields
the formula
\begin{equation}\label{eq:1}
  \mathsf{f}_{\theta}=\left|\ext\left(\theta\right)\right|-\left|\bar{r}\right|-1.
\end{equation}
It is worth noting that given the frequency of $\theta$ we can recover
its step and number of rests. The step is the quotient of the integer
division of $n$ by $\mathsf{f}_{\theta}$, and the number of rests
is the remainder.

Given $a\in\left[0,n\right]$ and $\theta$ a congruence of $\mathbf{L}_{n}$
let $\theta_{a}$ be the restriction of $\theta$ to the interval
$\left[0,a\right]$. That is, 
\[
\theta_{a}\doteq\theta\cap\left(\left[0,a\right]\times\left[0,a\right]\right).
\]
It is easily checked using Theorem \ref{th:charact-congruences} that
if $a$ is an extreme of $\theta$ that is not the right part of a
rest then $\theta_{a}$ is a congruence of $\mathbf{L}_{a}$.

We conclude this section with some facts concerning frequencies needed
in the sequel.
\begin{lem}
\label{lem:frecuencias}Let $\theta={\conr{k}{r}}$ and $\delta={\conr{l}{s}}$
be such that their join is not trivial and let $e_{1}$ be the first
positive element in $\ext(\theta)\cap\ext(\delta)$. Then:
\begin{enumerate}
\item $\mathsf{f}_{\theta\wedge\delta}=\left|\ext\left(\theta\right)\cap\ext\left(\delta\right)\right|-\left|\bar{r}\cap\bar{s}\right|-1$.
\item $\mathsf{f}_{\theta}=\mathsf{f}_{\theta_{e_{1}}}\mathsf{f}_{\theta\wedge\delta}$.
\item $\left(\theta\vee\delta\right)_{e_{1}}=\theta_{e_{1}}\vee\delta_{e_{1}}$.
\item $\mathsf{f}_{\theta\vee\delta}=\mathsf{f}_{\theta_{e_{1}}\vee\delta_{e_{1}}}\mathsf{f}_{\theta\wedge\delta}$.
\end{enumerate}
\end{lem}
\begin{proof}
1. This is immediate from Theorem \ref{thm:meet}.

\noindent 2. Remark \ref{rem:periodicidad} says that the folding
given by $\theta$ is composed of successive reflections of the folding
given by $\theta_{e_{1}}$. As each of these reflections runs between
consecutive common extremes that do not constitute a common rest,
the number of symmetrical pieces making up $\theta$ is 
\[
\left|\ext\left(\theta\right)\cap\ext\left(\delta\right)\right|-\left|\bar{r}\cap\bar{s}\right|-1.
\]

\noindent Each of these pieces contributes the same amount of upward
and downward slopes, namely $\mathsf{f}_{\theta_{e_{1}}}$. Thus,
$\mathsf{f}_{\theta}$ is the number in the display above times $\mathsf{f}_{\theta_{e_{1}}}$,
and the desired formula follows from 1.

\noindent 3. We prove the nontrivial inclusion, that is from left
to right. Suppose $\left\langle a,b\right\rangle \in\left(\theta\vee\delta\right)_{e_{1}}$.
Then there are $x_{1},\ldots,x_{m}\in\left[0,n\right]$ such that
\[
a\thr x_{1}\delr x_{2}\thr\cdots\thr x_{m}\delr b.
\]
Now, by Theorem \ref{thm:meet} $e_{1}$ is the step of $\theta\wedge\delta$,
and so $0/\theta\wedge\delta,\ldots,e_{1}/\theta\wedge\delta$ are
all the $\theta\wedge\delta$-blocks. In particular, there are $x'_{1},\ldots,x'_{m}\in\left[0,e_{1}\right]$
such that $x_{j}\mathrel{\left(\theta\vee\delta\right)}x'_{j}$ for
$j=1,\ldots,m$. So we have 
\[
a\thr x'_{1}\delr x'_{2}\thr\cdots\thr x'_{m}\delr b,
\]
which implies $\left\langle a,b\right\rangle \in\theta_{e_{1}}\vee\delta_{e_{1}}$.

\noindent 4. Note that $\theta\wedge\delta$ and $\theta\vee\delta$
have a nontrivial join. Also, as $\theta\wedge\delta\leq\theta\vee\delta$,
the first positive common extreme of these two congruences is the
step of $\theta\wedge\delta$, which equals $e_{1}$ by Theorem \ref{thm:meet}.
So 2 produces 
\[
\mathsf{f}_{\theta\vee\delta}=\mathsf{f}_{\left(\theta\vee\delta\right)_{e_{1}}}\mathsf{f}_{\theta\wedge\delta},
\]
and we apply 3 to yield the equality we set out to prove.
\end{proof}
%
%
%
%
%

%
%
\subsection{Characterization of the order of $\Con \tira{L}$}
Our next result is a characterization of the order relation in the
lattice of congruences of a line. It is also the first step in
understanding when two congruences have a nontrivial join.

To achieve this goal, we provide an alternative
characterization of the periodicity alluded to in
Remark~\ref{rem:periodicidad}. As it was seen in
Theorem~\ref{thm:meet}, the ``skeleton'' that holds both congruences in
Figure~\ref{fig:periodicas} is actually their meet: the points at which
the structure of both 
$\th$ and $\del$ is symmetric are the extremes of $\ga=\th\y\del$. 

Take the
case of $\th$, the congruence of the lower part of the figure. What we
just stated  amounts to say that the extremes of $\ga$ are extremes of
$\th$ and the structure of rests is copied from one slope of $\ga$ to
the next. So we say
that the rests of a congruence $\th$ are \emph{compatible} with $\ga$
if  for every $r\notin\ext(\ga)$ that is part of rest of
$\th$, and $r' \gar r$,  then $r'$ is part of a rest of $\th$.
\begin{theorem}\label{th:order}
  Assume $\th={\conr{k}{r}}$ and $\ga={\conr{e}{t}}$ are nontrivial congruences of a line
  $\tira{L}$. Then $\ga\leq\th$ if and only if
  \begin{enumerate}
  \item $e\in\ext(\th)$,%
  \item\label{item:2} $\bar t \subseteq\br$, and %
  \item\label{item:3} the rests of  $\th$ are compatible with $\ga$.
  \end{enumerate}
\end{theorem}
\begin{proof}
  ($\ent$) Assume $\ga\subseteq\th\neq L\times L$; then $\ga=\ga\y\th$. We
  may then apply Theorem~\ref{thm:meet} and similar arguments to show
  all three conditions hold.

  ($\tne$) We will see that for all $x\in [0,e]$, the $\ga$-class of
  $x$ is included in its $\th$-class.   Let $y>e$ such that $x\gar
  y$. Since $\displ{t}(x) =0$,  there exists a $d>0$ such that 
  \begin{equation}\label{eq:6}
    y - \displ{t}(y) = \pm x + 2d\cdot e.
  \end{equation}
  We  need the following:
  \begin{claim*}\label{cl:stronger}
    $\displ{r}(y) = 2d\cdot \displ{r}(e) + \displ{t}(y) \pm \displ{r}(x).$
  \end{claim*}
  We prove the Claim in the case that  the ``$\pm$'' is a ``$-$''.
  Our first step is to show that compatibility (Item~\ref{item:3}) implies
  that 
  the rests of $\th$ in the
  interval $[0,e)$ determine the remaining rests.

  Assume 
  \def\rs{r_s}
  $s\notin\ext(\ga)$, and let  $\rs$ be the  unique
  element of $[0,e)$ such that $s\conr{e}{t}\rs$; by definition of $\conr{e}{t}$
  there exists $b$ such that either
  \begin{enumerate}[(I)]
  \item \label{item:Caso_1}$s - \displ{t}(s) = \rs + 2b\cdot e$, or  
  \item \label{item:Caso_2} $    s - \displ{t}(s) =  -\rs + 2b\cdot e$
  \end{enumerate}
  hold. We claim that %
  \begin{equation}\label{eq:casos}
    \text{$s\in\br$ $\iff$ [$\rs\in\br$ in case~(\ref{item:Caso_1}) or
      $\rs-1\in\br$ in case~(\ref{item:Caso_2})].}
  \end{equation}

  For $(\ent)$, assume by way of contradiction that 
  there  exists $s\in\br$ falsifying (\ref{eq:casos}), and take the one with
  minimal $\rs$. Assume that $s$ satisfies~(\ref{item:Caso_1}); hence
  $\rs\notin\br$. Since by compatibility $\rs$ must be part of a rest,
  we conclude that $\rs-1\in\br$, and in particular $\rs-1>0$. 
  By subtracting 1 from Equation~(\ref{item:Caso_1}), and considering that
  $\displ{t}(s) =\displ{t}(s-1)$ (since $s-1$ cannot be a rest since
  $\ga$ is nontrivial), we obtain
  \[ s -1 - \displ{t}(s-1) = \rs-1 + 2b\cdot e, \]
  and hence $s-1 \conr{e}{t} \rs-1$. By compatibility $s-1$ must be
  part of a rest, therefore $s-2\in\br$. Now a  similar reasoning
  yields  $\rs-2\notin \br$ and $\rs-2$ is part of rest,
  contradicting our choice of $s$ as the counterexample with the
  minimal $\rs$. Case~(\ref{item:Caso_2}) and the $(\tne)$ direction are similar.
  
  We conclude in particular that the number of $\th$-rests between two
  consecutive elements of $\ext(\ga)$ not related by $\ga$ is constant
  and equals $\displ{r}(e)$. Then, taking $f$ as the greatest
  extreme of $\ga$ such that $f\leq y$, we easily  have
  \begin{equation*}
    \displ{r}(f) = (2d-1)\cdot\displ{r}(e) + \displ{t}(f) =  (2d-1)\cdot\displ{r}(e) + \displ{t}(y),
  \end{equation*}
  where the last equality holds by definition of $f$ and item~\ref{item:2}.
  Now $\displ{r}(y)$, the number of $\th$-rests before $y$, is equal to
  the number of rests before $f$ plus the number of rests between $f$
  and $y$. 
  \begin{equation*}
    \displ{r}(y) = (2d-1)\cdot\displ{r}(e) + \displ{t}(y) +\card{\{\text{$\th$-rests between $f$
        and $y$}\}}.
  \end{equation*}
  By using compatibility one more time we realize that the
  last summand equals the number of $\th$-rests between $x$ and $e$, that is
  $\displ{r}(e) -\displ{r}(x)$. So we finally obtain the Claim.

  The case where the ``$\pm$'' is a ``$+$'' is handled very similarly; in
  this case, $\displ{r}(f) = 2d\cdot\displ{r}(e)  + \displ{t}(y)$ and
  $\card{\{\text{$\th$-rests between $f$ and $y$}\}} =
  \displ{r}(x)$. This finishes the proof of the Claim.
  \smallskip

  We show next that $x \conr{k}{r} y$, i.e.,
  \begin{equation}\label{eq:7}
    y - \displ{r}(y) \con{k} x -\displ{r}(x)
  \end{equation}
  Let us work from the left hand side; substitute $\displ{r}(y)$
  according to the Claim:
  \begin{align*}
    y - \displ{r}(y) & = y -2d\cdot\displ{r}(e) - \displ{t}(y)
    \mp \displ{r}(x) \\
    & = \pm x   \mp \displ{r}(x) + 2d\cdot e -2d\cdot \displ{r}(e)  && \text{by Eq.~(\ref{eq:6})}\\
    & = \pm (x  - \displ{r}(x))  + 2d\cdot (e - \displ{r}(e)).
  \end{align*}
  But now, since $e\in\ext(\th)$, $e - \displ{r}(e)$ is congruent
  to $0$ or $k$ modulo $2k$, hence it is a multiple of $k$. Thus we have
  proved (\ref{eq:7}).

\end{proof}

%
%
%
%
%

%
\section{The case of Frequency Two}
\label{sec:case-frequency-two}

We now obtain the main results of the paper restricted to the case
when one of the congruences has frequency equal to 2. We need one more
concept to handle this case.

We say that a congruence $\conr{k}{r}$ of $\tira{L}_n$ is 
\emph{mirrored} provided that  for
all $r$ 
\[r\in\bar r \text{ if and only if } n-r-1\in \bar r.\] 
It is easy to see that a congruence is mirrored if and only if the folding given by it   
is symmetric with respect to its middle point. 

Observe that for a mirrored $\conr{k}{r}$  we have, for all $a$
\begin{equation}\label{eq:symmetry}
  \displ{r}(a)+\displ{r}(n-a) = \card{\br}.
\end{equation}
By applying this to different $b$ and $c$ and equating, we conclude
that 
\begin{equation}\label{eq:2}
  \displ{r}(b) - \displ{r}(c) = \displ{r}(n-c) -\displ{r}(n-b).
\end{equation}
We use this concept to study joins $\th\o\del$ in which
$\alt_\del=2$. It is immediate that there are two cases for this; let
$l=\bigl\lfloor\tfrac{n}{2}\bigr\rfloor$. If $n$ is
even, then  $\del ={\con{l}}$, and if $n$ is odd, then  $\del
={\con{l;l}}$. In either case we have that  $a\delr b$ if and only if
$a=b$ or $a+b=n$.

\begin{lemma}\label{lem:symmetric}
  Assume $\alt_\del=2$. 
  \begin{enumerate}
  \item \label{item:mirrored-permute} If $\th={\conr{k}{r}}$ is mirrored,
  $\th$ and $\del$ permute.
  \item \label{item:not-mirror-trivial} If $\th$ is not mirrored, then $\th\o\del$ is
  trivial.
  \end{enumerate}
\end{lemma}
\begin{proof}
  For item~\ref{item:mirrored-permute}, assume $a \delr b \conr{k}{r}c$ with $a\neq b$. Hence $a =
  n-b$. Let $x\doteq n-c$. Immediately $c\delr x$; it is enough to
  see that $x\conr{k}{r} a$. By definition this is $x- \displ{r}(x)
  \con{k} a -\displ{r}(a)$, i.e.,
  \begin{equation}\label{eq:3}
    n-c- \displ{r}(n-c)  \con{k} n-b -\displ{r}(n-b).
  \end{equation}
  
  We have two cases, according to the definition of $ b
  \conr{k}{r}c$. First consider that $c -\displ{r}(c) \modkk
  b-\displ{r}(b)$. We operate as follows:
  \begin{align*}
    c -\displ{r}(c) \modkk   b-\displ{r}(b) \iff \quad &\\ 
        c + (\displ{r}(b) - \displ{r}(c)) & \modkk b \\
        c + (\displ{r}(n-c) -\displ{r}(n-b))& \modkk b && \text{by
      Eq.~(\ref{eq:2})}\\
        -b -\displ{r}(n-b) & \modkk -c  -  \displ{r}(n-c)\\
        n-b -\displ{r}(n-b) & \modkk n-c  -  \displ{r}(n-c),
  \end{align*}
  and this implies (\ref{eq:3}).
  
  The other case is when $c -\displ{r}(c) \modkk  -b+\displ{r}(b)$.
  We have
  \begin{align*}
    c -\displ{r}(c)  \modkk   -b+\displ{r}(b) \iff \quad& \\ 
        -\displ{r}(b) - \displ{r}(c) &\modkk -c -b  \\
        -\displ{r}(b) - \displ{r}(c) &\modkk n -c +n -b -
    2\card{\br} && \text{since $k\mid n-\card{\br}$}  \\
        \card{\br} -\displ{r}(b) +\card{\br} -\displ{r}(c) &\modkk n -c +n -b 
    \\ 
        \displ{r}(n-b) + \displ{r}(n-c) &\modkk n -c +n -b &&
    \text{by Eq.~(\ref{eq:symmetry})}    \\ 
         - (n-b -\displ{r}(n-b))  &\modkk n-c  -  \displ{r}(n-c),
  \end{align*}
  and this also implies (\ref{eq:3}).

  Now we turn to the proof of item~\ref{item:not-mirror-trivial}.
  Suppose that $\th$ is not mirrored. Let $r$ be the least element of
  $L$ such that the equivalence $r\in 
  \bar r \iff n-r-1\in \bar r$ is false. Assume first that for this
  $r$,  $r =r_j\in   \bar r$ but $ n-r-1\notin \bar r$.
  
  Since $r\in \bar r$, $0<r<n-1$, and hence $0\leq n-r-2, n-r \leq
  n$. By the observations prior
  to Lemma~\ref{lem:symmetric} we know that $n-r\delr r$ and $n-r-1\delr r+1$, and since
  $r\in\br$ implies $r \thr r+1 $, we conclude that 
  \[n-r\delr r \thr r+1 \delr n-r-1,\]
  We claim that either  $n-r-1 \thr
  n-r+1$ or  there  is a rest at $n-r$. This would show  $\th\o\del$ 
  trivial by Lemma~\ref{lem:carac-trivial} since $n-r-1$, $n-r$ and
  $n-r+1$ are related by this  congruence, and we are done.
  
  To see the claim, suppose there is no a rest at $n-r$. Then we have
  $\displ{r}(n-r-1)  = \displ{r}(n-r+1) = m-j+1$ where $m=\card{\br}$,
  by the minimality of $r=r_j$ and because $n-r-1\notin  \br$.
  By expanding the definition of $n-r-1 \conr{k}{r} n-r+1$, it
  suffices to prove
  \begin{equation*}\label{eq:4}
    n-r-1 -\displ{r}(n-r-1)  \modkk   -(n-r+1)+\displ{r}(n-r+1).
  \end{equation*}
  In our context this is equivalent to  $n-r-1 + (n-r+1) - 2(m-j+1)\modkk 0$ and
  then to $ k \mid n - r_j - m +j-1$. But this  follows immediately
  from Theorem~\ref{th:charact-congruences}.

  The case in which $r\notin   \bar r$ and  $ n-r-1\in \bar r$ for the
  minimal $r$ is entirely analogous.  
\end{proof}

\begin{corollary}
  Assume that $\th,\del\in\Con \tira{L}_n$ and $\alt_\del=2$. If
  $\th\o\del$ is nontrivial then $\th$ and $\del$ permute.
\end{corollary}
%
%
%
%
%
%

\section{Trajectories}
\label{sec:trajectories}

We introduce next a practical representation of pairs of congruences
that will be of great help in understanding the join.
\begin{notation}
  In the sequel, $\tira{L} =
  \<L, R\> =\<\{0,\dots,n\}, R\>$ will denote a line (with $n\geq 1$)
  and $\th={\conr{k}{r}}$ 
  and $\del={\conr{l}{s}}$ will be nontrivial congruences of $\tira{L}$ with $k\leq
  l$. 
\end{notation}

Consider a rectangle of base $l$ and height $k$, that for
convenience can be regarded as sitting on the Euclidean plane and having
the origin and points of coordinates $(0,k)$, $(l,k)$, and $(l,0)$ as
vertices. Now we represent each element $x$ of $\tira{L}$ as a point
$P(x)$ with coordinates  $\bigl(P_1(x),P_2(x)\bigr)= (\min x/\del, \min x/\th)$, and for each
$0\leq x <n$, we connect the points $P(x)$ and $P(x+1)$ with a line
segment.

The resulting configuration of points and segments is called the
\emph{trajectory diagram of $\th\o\del$}. By construction, 
for any
$0\leq x,y \leq n$, $x\mathrel{(\th\o\del)}y$ if and only if there
exists a sequence $x_0,\dots,x_j$ of elements of $\tira{L}$   such
that $x_0=x$, $x_j=y$, and for each $0\leq i <j$, $P(x_i)$ and $P(x_{i+1})$ have
the same abscissa or the same ordinate. The uses of trajectory diagrams
in this paper will be restricted to congruences having $\ext(\th) \cap
\ext(\del) = \{0,n\}$, although they have been a key tool in understanding
the  general case.

\definecolor{eqdfdf}{rgb}{0.88,0.87,0.87}
\definecolor{wwwwww}{rgb}{0.6,0.6,0.6}
\begin{figure}[h]
  \begin{center}
    \begin{tikzpicture}[line cap=round,line join=round,>=triangle 45,x=1cm,y=1cm]
      \draw [color=eqdfdf,dash pattern=on 1pt off 1pt, xstep=1cm,ystep=1cm] (1,1) grid (7,5);
      \draw [line width=0.4pt,color=wwwwww] (1,1)-- (1,5);
      \draw [line width=0.4pt,color=wwwwww] (1,5)-- (7,5);
      \draw [line width=0.4pt,color=wwwwww] (7,5)-- (7,1);
      \draw [line width=0.4pt,color=wwwwww] (7,1)-- (1,1);
      \draw (1,1)-- (3,3)-- (5,5)-- (6,5)-- (7,4)-- (4,1)-- (1,4)-- (2,5)-- (3,5)-- (7,1);
      \fill [color=black] (1,1) circle (1.5pt);
      \draw[color=black] (0.85,1.16) node {$0$};
      \fill [color=black] (2,2) circle (1.5pt);
      \draw[color=black] (1.92,2.24) node {$1$};
      \fill [color=black] (3,3) circle (1.5pt);
      \draw[color=black] (2.98,3.28) node {$2$};
      \fill [color=black] (4,4) circle (1.5pt);
      \draw[color=black] (3.5,4.03) node {$3$};
      \fill [color=black] (5,5) circle (1.5pt);
      \draw[color=black] (5.07,5.23) node {$4$};
      \fill [color=black] (6,5) circle (1.5pt);
      \draw[color=black] (6.06,5.2) node {$5$};
      \fill [color=black] (7,4) circle (1.5pt);
      \draw[color=black] (7.15,4.08) node {$6$};
      \fill [color=black] (6,3) circle (1.5pt);
      \draw[color=black] (5.9,3.17) node {$7$};
      \fill [color=black] (5,2) circle (1.5pt);
      \draw[color=black] (4.92,2.19) node {$8$};
      \fill [color=black] (4,1) circle (1.5pt);
      \draw[color=black] (4,1.3) node {$9$};
      \fill [color=black] (3,2) circle (1.5pt);
      \draw[color=black] (3.11,2.24) node {$10$};
      \fill [color=black] (2,3) circle (1.5pt);
      \draw[color=black] (2.08,3.25) node {$11$};
      \fill [color=black] (1,4) circle (1.5pt);
      \draw[color=black] (1.45,4.05) node {$12$};
      \fill [color=black] (2,5) circle (1.5pt);
      \draw[color=black] (2.08,5.23) node {$13$};
      \fill [color=black] (3,5) circle (1.5pt);
      \draw[color=black] (3.14,5.22) node {$14$};
      \fill [color=black] (4,4) circle (1.5pt);
      \draw[color=black] (4.45,4.06) node {$15$};
      \fill [color=black] (5,3) circle (1.5pt);
      \draw[color=black] (5.21,3.17) node {$16$};
      \fill [color=black] (6,2) circle (1.5pt);
      \draw[color=black] (6.2,2.19) node {$17$};
      \fill [color=black] (7,1) circle (1.5pt);
      \draw[color=black] (7.22,1.08) node {$18$};
    \end{tikzpicture}
  \end{center}
  \caption{Trajectory diagram of $\con{4;4,13}\o\con{6}$ on $\tira{L}_{18}$.} \label{fig:trajectory}
\end{figure}
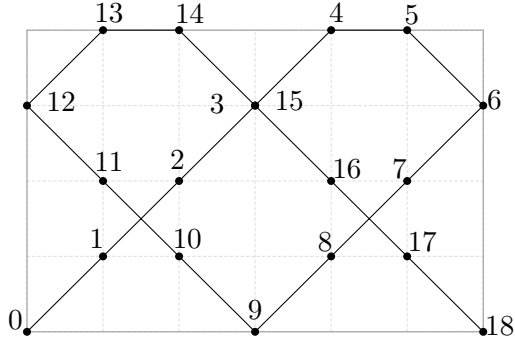

It may be considered that the diagram pictures the trajectory of a
particle that starting at the origin $P(0)$ moves through the line
$\tira{L}$, bouncing on the horizontal borders of the rectangle at
elements $e$ of $\ext(\th)\setminus\{0,n\}$ and on vertical borders at elements of
$\ext(\del)\setminus\{0,n\}$. There are two types of \emph{bounces}, depending whether
$e$ is part of a rest or not. The part of the trajectory that
lies in the interior of the rectangle consists of line segments with
slope  $\pm 1$. We will see that trajectories picturing nontrivial
joins do not have ``overlapping'' segments, i.e.,
self-intersections consist in isolated points. It is  clear that such intersections
 occur only at points with (half-) integral
coordinates. We call \emph{crossings} such intersections in the
interior of the rectangle. 

When a crossing occurs at a point $Q$ with integral coordinates, then
there must exist two elements of the line $a$ and $d$ such that
$Q=Q(a)=Q(d)$, and their immediate successors and predecessors satisfy
either 
\begin{equation}\label{eq:8}
  d\pm 1 \thr a-1 \delr d\mp 1 \thr a+1 \delr d\pm 1
\end{equation}
or
\begin{equation}\label{eq:9}
  d\pm 1 \delr a-1 \thr d\mp 1 \delr a+1 \thr d\pm 1.
\end{equation}
Similarly, when a crossing occurs at a point $P$ with half-integral
coordinates, there must be elements $b,c$ satisfying either
\[
b \delr c \thr b\pm 1 \delr c\pm 1 \thr b
\]
or
\[
b \thr c \delr b\pm 1 \thr c\pm 1 \delr b
\]
where $|P_i-P_i(b)| = |P_i-P_i(c)| = \tfrac{1}{2}$ for $i=1,2$.

We will usually speak in geometrical terms when referring to diagrams,
as in \emph{there is rest at the bounce at $k$} and the like.  

\begin{lemma}\label{lem:dist-un-medio}
Suppose that in the trajectory diagram of $\th\o\del$ one of the following
happens:
\begin{itemize}
\item There are two crossings at
  points $P=(P_1,P_2)$ and  $Q=(Q_1,Q_2)$ such that $|P_i -
  Q_i|=\tfrac{1}{2}$ for some
  $i=1,2$,
\item there is a crossing at $P$ such that there is $x\in \ext(\th)\cup\ext(\del)$ which is
  not part of  a rest and $|P_i(x)-P_i|=\tfrac{1}{2}$ for some
  $i=1,2$, or 
\item  there is a crossing at $P$ such that there is $x\in L$ which is
  part of  a rest of either $\th$ or $\delta$ and $P_i(x)=P_i$ for some
  $i=1,2$.
\end{itemize}
Then the join is trivial.
\end{lemma}
More intuitively: If the trajectory diagram of $\th\o\del$ has two crossings with one
coordinate differing in $\tfrac{1}{2}$, the join is trivial. And the same
applies to a crossing with one coordinate differing in $\tfrac{1}{2}$
from ``the center'' of a bounce. An example of the first situation is
depicted in Figure~\ref{fig:diff_one_half}.
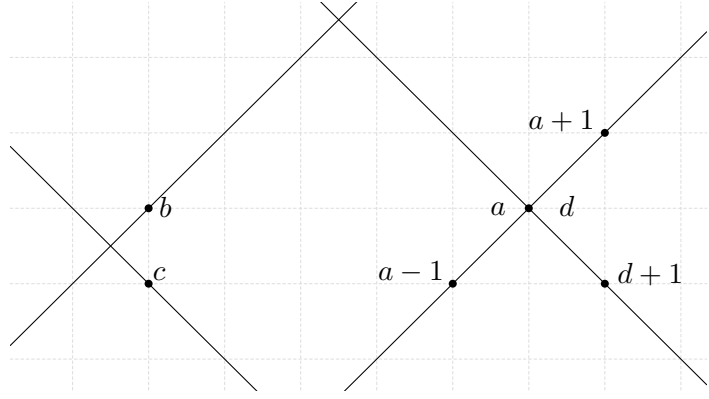
\begin{figure}
  \begin{center}
    \definecolor{uququq}{rgb}{0.25,0.25,0.25}
    \definecolor{eqdfdf}{rgb}{0.88,0.87,0.87}
    \begin{tikzpicture}[line cap=round,line join=round,>=triangle 45,x=1cm,y=1cm]
      \draw [color=eqdfdf,dash pattern=on 1pt off 1pt, xstep=1cm,ystep=1cm] (10.18,7.59) grid (19.34,12.73);
      \clip(10.18,7.59) rectangle (19.34,12.73);
      \draw [domain=10.18:19.34] plot(\x,{(--21-1*\x)/1});
      \draw [domain=10.18:19.34] plot(\x,{(--2-1*\x)/-1});
      \draw [domain=10.18:19.34] plot(\x,{(-14--2*\x)/2});
      \draw [domain=10.18:19.34] plot(\x,{(--54-2*\x)/2});
        \fill [color=black] (12,9) circle (1.5pt);
        \draw[color=black] (12.14,9.13) node {$c$};
        \fill [color=black] (12,10) circle (1.5pt);
        \draw[color=black] (12.22,10.02) node {$b$};
        \fill [color=black] (16,9) circle (1.5pt);
        \draw[color=black] (15.45,9.16) node {$a-1$};
        \fill [color=black] (18,11) circle (1.5pt);
        \draw[color=black] (17.43,11.16) node {$a+1$};
        \fill [color=black] (18,9) circle (1.5pt);
        \draw[color=black] (18.6,9.13) node {$d+1$};
        \fill [color=black] (17,10) circle (1.5pt);
        \draw[color=black] (16.6,10) node {$a$};
        \draw[color=black] (17.5,10.03) node {$d$};
    \end{tikzpicture}
  \end{center}
\caption{Two crossings with a  $\tfrac{1}{2}$
  difference.} \label{fig:diff_one_half}
\end{figure}
\begin{proof}
We only prove the first item. 
Without loss of generality, assume that $Q$ has integral coordinates.
Consider the case when $|P_2 - Q_2|=\tfrac{1}{2}$. By the observation
prior to the statement of the Lemma, let $a,d\in L$ satisfying one of
the equations (\ref{eq:8}) or (\ref{eq:9}). We might also find $b,c\in L$ such
that  $|P_i-P_i(b)|
= |P_i-P_i(c)| = \tfrac{1}{2}$ for $i=1,2$,  $b\thr a$ and $b\delr c$.
There are two possibilities for the relative position of  the two
crossings:
\begin{enumerate}
\item $c\thr a-1 \delr d +\epsilon$ (where $\epsilon=\pm 1$): In this
  sub-case, we have
  \[
  d+\epsilon \thr a-1 \delr d-\epsilon \thr a+1 \delr d+\epsilon.
  \]
  Hence we have the following chain of relations
  \[
  a\thr b \delr c \thr a-1 \delr  d + \epsilon \delr a+1,
  \]
  which witnesses that $a-1,a,a+1$ all belong to the same
  $\th\o\del$-class, hence the join is trivial.
\item $c\thr a+1 \delr d +\epsilon$ for some $\epsilon=\pm 1$: We
  deduce
  \[
  d+\epsilon \delr a-1 \thr d-\epsilon \delr a+1 \thr d+\epsilon
  \]
  and obtain
  \[
  a\thr b \delr c \thr a+1 \delr  d + \epsilon \delr a-1,
  \]
  reaching triviality once again.
\end{enumerate}
The case where  $|P_1 - Q_1|=\tfrac{1}{2}$ is completely analogous.
\end{proof}
\begin{lemma}\label{lem:menor-k}
  Assume that the trajectory of $\th\o\del$ has two bounces at $x$ and
  $y$ on the same side of the diagram, such that:
  \begin{enumerate}
  \item $y$ neither lies at  a  corner nor is part of a rest, 
  \item $x$ is part of a rest, and 
  \item the distance between the bounces is less than $k$.
  \end{enumerate}
  Then the  join  $\th\o\del$  is trivial.
\end{lemma}
\begin{proof}
  We consider first the case of bounces on a horizontal
  border. Without loss of generality, we may assume that $y$ is
  pictured to the left of $x$, and that $x$ is the part of the rest
  farthest from $y$ (i.e., $\{x,x'\}$ constitute a rest and $P_1(x) = P_1(x')+1$).

  Let  $a$  be the closest point to $y$ in the line having
  $P_1(a)=P_1(x)$. If the bounces are indeed adjacent (distance
  0), then $y$ and $a$ are consecutive elements of the line (see
  Figure~\ref{fig:adjacent-bounces}). Let $c$ be the other point on
  the line at   distance 1 from $y$; hence $P(c)=(P_1(y)-1,1)$. We have $y \thr x \delr a
  \thr c$, hence the consecutive points $c$, $y$, and $a$ are
  $(\th\o\del)$-related, hence $\th\o\del$  is trivial.
  \begin{figure}[h]
    \begin{center}
      \definecolor{ccqqcc}{rgb}{0.8,0,0.8}
      \definecolor{qqccqq}{rgb}{0,0.8,0}
      \definecolor{wwwwww}{rgb}{0.4,0.4,0.4}
      \definecolor{zzzzzz}{rgb}{0.6,0.6,0.6}
      \begin{tikzpicture}[line cap=round,line join=round,>=triangle 45,x=0.77cm,y=0.77cm]
        \draw [color=eqdfdf,dash pattern=on 1pt off 1pt,
          xstep=0.77cm,ystep=0.77cm] (9.32,5) grid (12.84,6.9);
        \clip(9.32,4.5) rectangle (12.84,6.9);
        \draw [line width=0.4pt] (8,5)-- (27,5);
        \draw (19,13)-- (11,5)-- (7,9)-- (8,13);
        \draw (3,13)--  (11,5)-- (12,5)-- (20,13);
        \fill [color=black] (12,5) circle (1.5pt);
        \draw[color=black] (12,4.7) node {$x$};
        \fill [color=black] (11,5) circle (1.5pt);
        \draw[color=black] (11,4.7) node {$y$};
        \fill [color=black] (10,6) circle (1.5pt);
        \draw[color=black] (9.79,5.74) node {$c$};
        \fill [color=black] (12,6) circle (1.5pt);
        \draw[color=black] (12.34,5.81) node {$a$};
      \end{tikzpicture}
    \end{center}
    \caption{Adjacent bounces.} \label{fig:adjacent-bounces}
  \end{figure}
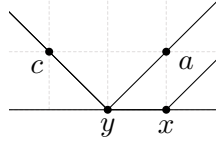
  
  If the distance is positive, we may assume without loss of
  generality that the situation at hand  is the one pictured in
  Figure~\ref{fig:closer-k}. The points $b,b',b''$ are the closest
  points   to $x$ such that $P(b)=(P_1(y)-1,P_2(a))$,
  $P(b')=(P_1(y),P_2(a)-1)$, and  $P(b'')=(P_1(y)+1,P_2(a)-2)$. These
  are well defined since $P(y)$ is not at a corner and $P_2(a)\geq 2$.
  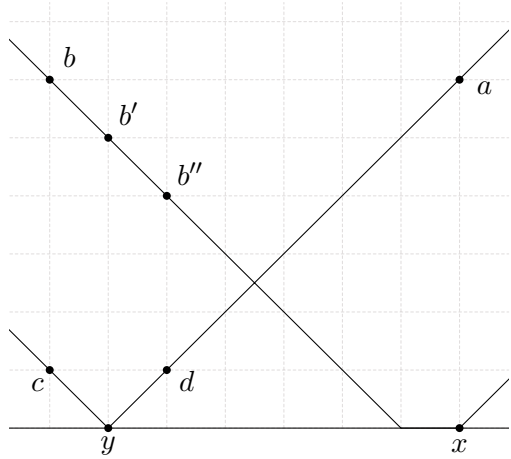
\begin{figure}
    \begin{center}
      \definecolor{ccqqcc}{rgb}{0.8,0,0.8}
      \definecolor{qqccqq}{rgb}{0,0.8,0}
      \definecolor{wwwwww}{rgb}{0.4,0.4,0.4}
      \definecolor{zzzzzz}{rgb}{0.6,0.6,0.6}
      \begin{tikzpicture}[line cap=round,line join=round,>=triangle 45,x=0.77cm,y=0.77cm]
        \draw [color=eqdfdf,dash pattern=on 1pt off 1pt, xstep=0.77cm,ystep=0.77cm] (9.32,5) grid (17.84,12.33);
        \clip(9.32,4.5) rectangle (17.84,11.9);
        \draw [line width=0.4pt] (8,5)-- (27,5);
        \draw (20,13)-- (18.99,13)-- (11,5)-- (7,9)-- (8,13)-- (8,13)-- (16,5);
        \draw  (16,5)-- (17,5)-- (25,13);
        \fill [color=black] (17,11) circle (1.5pt);
        \draw[color=black] (17.43,10.86) node {$a$};
        \fill [color=black] (17,5) circle (1.5pt);
        \draw[color=black] (17,4.7) node {$x$};
        \fill [color=black] (11,5) circle (1.5pt);
        \draw[color=black] (11,4.7) node {$y$};
        \fill [color=black] (11,10) circle (1.5pt);
        \draw[color=black] (11.35,10.42) node {$b'$};
        \fill [color=black] (10,11) circle (1.5pt);
        \draw[color=black] (10.34,11.39) node {$b$};
        \fill [color=black] (10,6) circle (1.5pt);
        \draw[color=black] (9.79,5.74) node {$c$};
        \fill [color=black] (12,6) circle (1.5pt);
        \draw[color=black] (12.34,5.81) node {$d$};
        \fill [color=black] (12,9) circle (1.5pt);
        \draw[color=black] (12.4,9.38) node {$b''$};
      \end{tikzpicture}
    \end{center}
    \caption{Different bounces closer than $k$ imply triviality.} \label{fig:closer-k}
  \end{figure}
  We obtain 
  \[a\thr b \delr c \thr d \delr b'' \quad\text{and}\quad
  a\delr x \thr y \delr b',\]
  hence the join is trivial.   All of the
  reasoning took place inside the rectangle having vertices $a$,
  $b$ and $x$, having  height $\overline{ax}\leq k$.

  For the case of the vertical border, one should only be careful with
  the case of the distance being exactly $k-1$. But then the bounce
  at $y$ lies on a corner, and the result holds vacuously.
\end{proof}
\begin{corollary}\label{cor:vert-iguales}
  If the join $\th\o\del$ is nontrivial, the bounces on the interior
  of each vertical border of its trajectory diagram are of the same type.
\end{corollary}

We work under the hypothesis that
$\th\o\del\neq L\times L$ up to the end of the present section.
\begin{lemma}\label{lem:autocruces}
  Assume that the trajectory of $\th\o\del$ bounces at $y$ on the interior of
  the left side of the diagram, and let $x$ and $z$ be the bounces
  next to $y$ on  the top side and bottom side, respectively. Then the
  bounce
  at $y$ is of the same type as the one at $z$.
\end{lemma}
\begin{proof}
  Assume without loss of generality that $x<y<z$. By Lemma~\ref{lem:menor-k}, we 
  have that $x$ is part of a rest if and only if  $k$ is a rest. Let
  $w$ the last element of the line   before $x$ such that
  $P_1(w)=k$. We may proceed to perform a case analysis. 
  
  Assume  $k$ is not a rest (see the left diagram in 
  Figure~\ref{fig:autocruces}). If both $y$ and $z$ are (not) parts of a rest, it can
  be shown that the element $w+2k+1$ ($w+2k$) is related to $w$ by
  both congruences. Hence they are pictured in the same point $P(w)$
  and correspond to a crossing of the trajectory, right under $P(k)$. 

  Now if $y$ is part of a rest and $z$ is not (upper thin dashed line
  in Figure~\ref{fig:autocruces}), $w\thr w+2k$ and
  $w\delr w+2k+1$; this can be seen analytically noticing that
  $\Delta_\br(w)=\Delta_\br(w+2k)$ and 
  $\Delta_\bs(w)+1=\Delta_\bs(w+2k+1)$. This implies that there is a
  crossing at coordinates
  $P(w)+(-\frac{1}{2},\frac{1}{2})=(k-\frac{1}{2},P_2(w)+\frac{1}{2})$,
  which differs from the bounce at $k$ by $\frac{1}{2}$ in the
  abscissa. Hence the join is trivial by
  Lemma~\ref{lem:dist-un-medio}. 
  
  The remaining cases follow analogously.  
  
  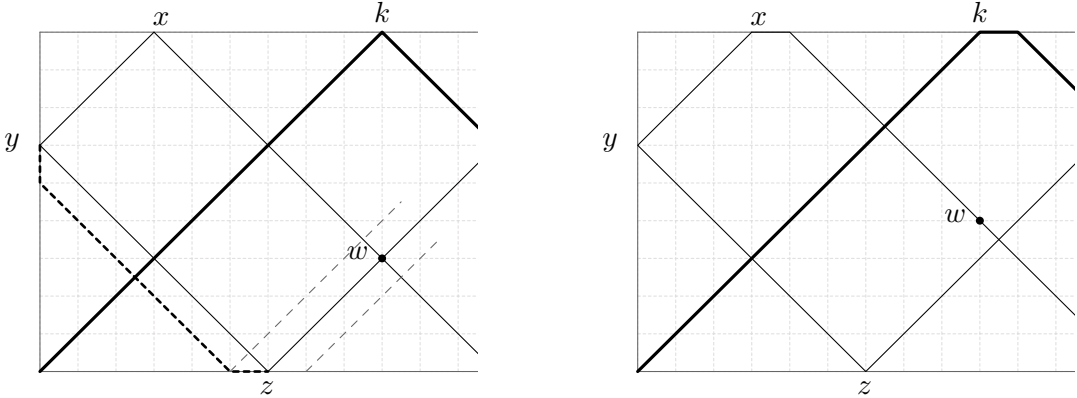
\begin{figure}[h]
    \begin{center}
      \definecolor{wwwwww}{rgb}{0.4,0.4,0.4}
      \definecolor{eqdfdf}{rgb}{0.88,0.87,0.87}
      \begin{tabular}{ccc}
        \begin{tikzpicture}[line cap=round,line join=round,>=triangle 45,x=0.5cm,y=0.5cm]
          \clip(7,4) rectangle (19.5,15);
          \draw [color=eqdfdf,dash pattern=on 1pt off 1pt,
            xstep=0.5cm,ystep=0.5cm] (8,5) grid (19.5,14);
          \draw [line width=0.4pt,color=wwwwww] (8,14)-- (8,5);
          \draw [line width=0.4pt,color=wwwwww] (8,5)-- (21,5);
          \draw [line width=0.4pt,color=wwwwww] (22,5)-- (27,5);
          \draw [line width=0.4pt,color=wwwwww] (27,14)-- (8,14);
          \draw [line width=1.2pt] (8,5)-- (17,14)-- (26,5);
          \draw (11,14)-- (8,11)-- (14,5)-- (23,14);
          \draw (11,14)-- (20,5);
          \draw [line width=1pt,dash
            pattern=on 2pt off 2pt] (8,11)-- (8,10)-- (13,5)-- (14,5);
          \draw [dash pattern=on 3pt off 3pt,color=wwwwww] (13,5)-- (17.5,9.5);
          \draw [dash pattern=on 3pt off 3pt,color=wwwwww] (15,5)-- (18.5,8.5);
          \draw (6.8,11.58) node[anchor=north west] {$y$};
          \fill [color=black] (17,8) circle (1.5pt);
          \draw (15.8,8.58) node[anchor=north west] {$w$};
          \draw (13.5,5) node[anchor=north west] {$z$};
          \draw (10.7,14.8) node[anchor=north west] {$x$};
          \draw (16.52,15.07) node[anchor=north west] {$k$};
        \end{tikzpicture}
        & 
        \hspace{2em}
        & 
        \begin{tikzpicture}[line cap=round,line join=round,>=triangle 45,x=0.5cm,y=0.5cm]
          \clip(7,4) rectangle (19.5,15);
          \draw [color=eqdfdf,dash pattern=on 1pt off 1pt,
            xstep=0.5cm,ystep=0.5cm] (8,5) grid (19.5,14);
          \draw [line width=0.4pt,color=wwwwww] (8,14)-- (8,5);
          \draw [line width=0.4pt,color=wwwwww] (8,5)-- (21,5);
          \draw [line width=0.4pt,color=wwwwww] (22,5)-- (27,5);
          \draw [line width=0.4pt,color=wwwwww] (27,14)-- (8,14);
          \draw [line width=1.2pt] (8,5)-- (17,14)-- (18,14)-- (27,5);
          \draw (11,14)-- (8,11)-- (14,5)-- (23,14);
          \draw (11,14)-- (12,14)-- (21,5);
          \draw (6.8,11.58) node[anchor=north west] {$y$};
          \fill [color=black] (17,9) circle (1.5pt);
          \draw (15.8,9.58) node[anchor=north west] {$w$};
          \draw (13.5,5) node[anchor=north west] {$z$};
          \draw (10.7,14.8) node[anchor=north west] {$x$};
          \draw (16.52,15.07) node[anchor=north west] {$k$};
        \end{tikzpicture}
      \end{tabular}
    \end{center}
    \caption{Correlation of bounces near the origin.}\label{fig:autocruces}
  \end{figure}
\end{proof}

%
%
%

%
%
%
%
%
%
%
%
%
%
%
%
%
%
%

%
%
%
%
%
%
%
%
%
%
%
%
%

%
\subsection{Bounces in Trajectory Diagrams}
We will now prove that under the assumption of $\alt_\th,\alt_\del\geq 3$, bounces
in each side of the trajectory diagram must be of the same type,
extending the previous lemmas.

By Corollary~\ref{cor:vert-iguales} we
only have to consider 4 cases; they are as follows:
\begin{enumerate}
\item $k$ is a not a rest of $\th$.
  \begin{enumerate}[a.]
  \item \label{item:1a}No rests on the left border.
  \item \label{item:1b}There are rests on the left border.
  \end{enumerate}
\item $k$ is a rest of $\th$. 
  \begin{enumerate}[a.]
  \item \label{item:2a}No rests on the left border.
  \item \label{item:2b}There are rests on the left border.
  \end{enumerate}
\end{enumerate}
\begin{lemma}\label{lem:alt-3-rebote}
If $\ext(\th)\cap \ext(\del) = \{0,n\}$ and $\alt_\del\geq 3$, there
exists $x\in 0/\del$ such that $x\notin\ext(\th)$.
\end{lemma}
\begin{proof}
Let $x\doteq \min (0/\del \setminus \{0\})$. Note that $x\neq n$,
otherwise we would have $\alt_\del=2$. By definition, $x\in
\ext(\del)$. If $x\in\ext(\th)$, we would have that $x\in\ext(\th)\cap
\ext(\del) = \{0,n\}$, an absurdity.
\end{proof}
\begin{lemma}\label{lem:bounces}
  If $\alt_\th,\alt_\del\geq 3$, bounces
  in each side of the trajectory diagram  of $\th\o\del$ must be of the
  same type.  
\end{lemma}
\begin{proof}
  We will work on  case \ref{item:2a}, i.e., $k$ is a rest of $\th$
  and there are no rests on the left border. The reader will easily
  note that the arguments for the other cases are completely similar.

  By Lemma~\ref{lem:alt-3-rebote}, there must be at least one bounce
  on the interior of the left side. Every such bounce determines unique bounces on the
  upper and lower side with abscissas smaller than $k$.
  Enumerate the distances between the upper bounces to the left of $k$ as
  $d_1,\dots,d_\al$ (hence $\al\geq 1$), and let $d_0$ be the distance
  from the left border to the leftmost bounce. The first $\al+1$ upper bounces
  are of the same type by Lemma~\ref{lem:menor-k}. 

  Note that the distances on the left and lower sides of the rectangle
  are as labeled since every segment in its  interior has slope $\pm
  1$. Hence, by
  Lemma~\ref{lem:autocruces}, the first $\al$  
  lower bounces are not rests. If there are no further upper bounces,
  we are done; otherwise, the next upper bounce (immediately to the
  right of $k$) must also be a rest, since $d_\al\leq k-2$.
  
  The bounce labeled as $x$ cannot be a rest (see
  Figure~\ref{fig:trajectory_diagram}). If it were, the distance 
  $h$ would equal $d_1$ and then the heights of crossings at $y$
  and $z$ would differ in $\tfrac{1}{2}$. This contradicts
  non-triviality of the join by Lemma~\ref{lem:dist-un-medio}. Hence
  $h=d_1+1$; moreover, since all segments of the trajectory in the
  interior of the rectangle have slope $\pm 1$, this allows us to
  conclude that the distances between the subsequent lower bounces
  repeat the pattern of the first $\alpha$.

  A similar argument enables us to conclude that the bounce at
  $x'$ in Figure~\ref{fig:trajectory_diagram} \emph{must be} a rest;
  otherwise the distance $h'$ would equal $h=d_1+1$ and then the
  crossing  labeled as $y'$ would be $\frac{d_1+1}{2}$ below the upper
  border of the diagram, but the crossing   at $z'$ is $\frac{1}{2}$
  above that, contradicting nontriviality by
  Lemma~\ref{lem:dist-un-medio}.
  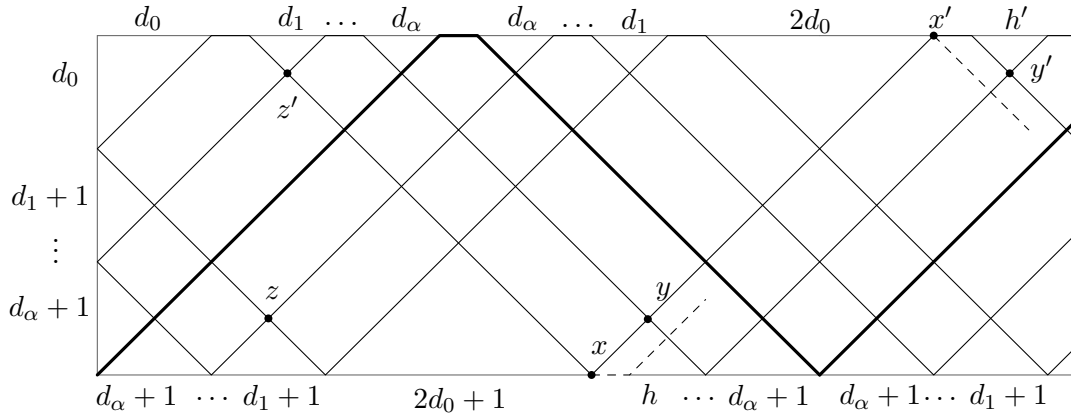
\begin{figure}
    \begin{center}
      \definecolor{wwwwww}{rgb}{0.4,0.4,0.4}
      \definecolor{eqdfdf}{rgb}{0.88,0.87,0.87}
      \begin{tikzpicture}[line cap=round,line join=round,>=triangle 45,x=0.5cm,y=0.5cm]
        \clip(5.7,4) rectangle (33.61,15);
        \draw [line width=0.4pt,color=wwwwww] (8,14)-- (8,5);
        \draw [line width=0.4pt,color=wwwwww] (8,5)-- (21,5);
        \draw [line width=0.4pt,color=wwwwww] (22,5)-- (27,5);
        \draw [line width=0.4pt,color=wwwwww] (27,14)-- (8,14);
        \draw [line width=1.2pt] (8,5)-- (17,14)-- (18,14)-- (27,5);
        \draw (11,14)-- (8,11)-- (14,5)-- (23,14);
        \draw (11,14)-- (12,14)-- (21,5);
        \draw (27,8)-- (21,14)-- (20,14)-- (11,5)-- (8,8)-- (14,14)-- (15,14)-- (24,5);
        \draw (21,5)-- (27,11);
        \draw (24,5)-- (27,8);
        \draw (8.7,15.07) node[anchor=north west] {$d_0$};
        \draw (12.48,15.07) node[anchor=north west] {$d_1$};
        \draw (15.47,15.04) node[anchor=north west] {$d_\alpha$};
        \draw (6.53,13.58) node[anchor=north west] {$d_0$};
        \draw (5.46,10.3) node[anchor=north west] {$d_1+1$};
        \draw (6.53,9.3) node[anchor=north west] {$\vdots$};
        \draw (5.4,7.38) node[anchor=north west] {$d_\alpha+1$};
        \draw (7.69,5) node[anchor=north west] {$d_\alpha+1$};
        \draw (11.56,5) node[anchor=north west] {$d_1+1$};
        \draw (16.05,4.94) node[anchor=north west] {$2d_0+1$};
        \draw (22.03,5.03) node[anchor=north west] {$h$};
        \draw (23.1,4.69) node[anchor=north west] {$\dots$};
        \draw (13.7,14.64) node[anchor=north west] {$\dots$};
        \draw (10.31,4.69) node[anchor=north west] {$\dots$};
        \draw (18.52,15.07) node[anchor=north west] {$d_\alpha$};
        \draw (21.51,14.98) node[anchor=north west] {$d_1$};
        \draw [line width=0.4pt,color=wwwwww] (27,14)-- (38,14);
        \draw [line width=0.4pt,color=wwwwww] (27,5)-- (38,5);
        \draw (27,8)-- (33,14)-- (34,14)-- (38,10);
        \draw (27,11)-- (30,14)-- (31,14)-- (38,7);
        \draw (23,14)-- (24,14)-- (33,5)-- (38,10);
        \draw [line width=1.2pt] (27,5)-- (36,14)-- (37,14)-- (38,13);
        \draw (27,8)-- (30,5)-- (38,13);
        \draw (27.25,5.03) node[anchor=north west] {$d_\alpha+1$};
        \draw (19.93,14.58) node[anchor=north west] {$\dots$};
        \draw (29.45,4.69) node[anchor=north west] {$\dots$};
        \draw (30.67,5.03) node[anchor=north west] {$d_1+1$};
        \draw (25.94,14.98) node[anchor=north west] {$2d_0$};
        \draw (31.59,15.04) node[anchor=north west] {$h'$};
        \draw [dash pattern=on 3pt off 3pt] (30,14)-- (32.5,11.5);
        \draw [dash pattern=on 3pt off 3pt] (21,5)-- (22,5)-- (24,7);
        \draw (24.32,5) node[anchor=north west] {$d_\alpha+1$};
          \fill [color=black] (21,5) circle (1.5pt);
          \draw[color=black] (21.21,5.73) node {$x$};
          \fill [color=black] (22.48,6.48) circle (1.5pt);
          \draw[color=black] (22.89,7.2) node {$ y$};
          \fill [color=black] (12.5,6.5) circle (1.5pt);
          \draw[color=black] (12.6,7.17) node {$ z$};
          \fill [color=black] (13,13) circle (1.5pt);
          \draw[color=black] (13,12.11) node {$z'$};
          \fill [color=black] (30,14) circle (1.5pt);
          \draw[color=black] (30.18,14.49) node {$x'$};
          \fill [color=black] (32,13) circle (1.5pt);
          \draw[color=black] (32.84,13.12) node {$y'$};
      \end{tikzpicture}
    \end{center}
    \caption{Trajectory diagram for case \ref{item:2a}.}\label{fig:trajectory_diagram}
  \end{figure}
  
  Observe that by the previous reasoning we can conclude that all the
  distances between two bounces in the diagram are either less than $k$
  and hence they are of the same type by Lemma~\ref{lem:menor-k};  or
  else they are equal to $2d_0+1$ (for lower bounces) or $2d_0$ (for
  upper bounces), and in this case we are in a position where
  Lemma~\ref{lem:dist-un-medio} can be applied. Therefore, by
  repeating the previous arguments we obtain the result.
\end{proof}

The previous argument can be made more formal (although painstakingly more
cumbersome) by considering an inductive argument, which we proceed to
sketch. It can be seen that in the diagram for a nontrivial
$\th\o\del$, if there is (not) a rest at $k$, every time
the trajectory crosses from right to left the vertical segment at abscissa
$k +\frac{1}{2}$ ($k$)  and bounces in the left border, it crosses
this segment in the same point 
in its way back. Hence if we disregard all the part of the diagram
with abscissas less than $k+1$ (less than $k$) if there is (not) a rest at
$k$, what remains is also a trajectory diagram (of eventually a
different type, e.g., chopping a diagram in case \ref{item:2a} at
$k+1$ will lead to a diagram in case \ref{item:1b} flipped vertically). Hence this
provides a way to pass to a smaller line.

%
%
%
%
%

\section{Main Results}
\label{sec:main-results}
\subsection{The catalog of nontrivial joins}
\label{sec:catalog-non-trivial-joins}
In this section we will give a complete description of the possible shapes
of trajectories picturing nontrivial joins. 

We will  say that a congruence $\lambda$ is even (resp.\ odd) if
$\alt_\lambda$ is even (resp.\ odd). It is immediate to see that
if $\th$ is even (odd), the trajectory of $\th\o\del$ finishes
at one of the lower (upper) corners of the diagram: that is, $P(n)$ is
equal to either $(0,0)$ or $(l,0)$ (resp., $(0,k)$ or
$(l,k)$). Likewise, if $\del$ is even (odd), the trajectory of $\th\o\del$ finishes
at one of the left (right) corners of the diagram. 

We will focus in the present subsection in the case $\ext(\th)\cap\ext(\del)
  =\{0,n\}$. Assume that the
join $\th\o\del$ is nontrivial and that $\alt_\th,\alt_\del\geq 3$. In view of
Lemma~\ref{lem:bounces}, we can classify the congruence $\th={\conr{k}{r}}$ according to it having no rests, having rests only on the top
(i.e., $r_i=(2k+1)i -k -1$ for each $i\leq \card{\br}$),  having rests only on the bottom
(i.e., $r_i=(2k+1)i-1$ for each $i\leq \card{\br}$), or having rests
everywhere. The same holds for $\del$.

\begin{lemma}\label{lem:autocruces-varios}
  Assume that  $\th\o\del$ is nontrivial, $\ext(\th)\cap\ext(\del)
  =\{0,n\}$, 
  and $\alt_\th,\alt_\del\geq 3$.
  \begin{enumerate}
  \item \label{selfcrossingevenodd} $\theta$ even and $\delta$ odd imply
    \textup[$\th$ has rests on bottom $\iff$ $\del$ has rests on top\textup].
  \item \label{selfcrossingoddodd} $\theta$ odd and $\delta$  odd
    imply 
    \textup[$\th$ has rests on top $\iff$ $\del$ has rests on top\textup].
  \item \label{selfcrossingoddeven}
    $\theta$ odd and $\delta$ even imply 
    \textup[$\th$ has rests on top $\iff$ $\del$ has rests on bottom\textup].
  \item For every $\theta$ and $\delta$, 
    \textup[$\th$ has rests on bottom $\iff$ $\del$ has rests on bottom\textup].
  \end{enumerate}
\end{lemma}
\begin{proof}
  For \ref{selfcrossingevenodd}, consider the trajectory diagram of
  $\th\o\del$. We may apply the proof of Lemma~\ref{lem:autocruces} but now looking 
  at the lower right corner, that is, where the trajectory finishes. We conclude that bounces on the right side are of the same type as those on the bottom side. By
  Lemma~\ref{lem:bounces}, we conclude that $\th$ has rests on bottom
  if and only if $\del$ has rests on top. The next two items follow
  by considering the upper right and upper left corners, respectively.

  The final item is simply the consequence of Lemma~\ref{lem:autocruces} under the light of Lemma~\ref{lem:bounces}.
\end{proof}

\begin{theorem}[The catalog]\label{th:catalog}
  Assume that  $\th\o\del\neq L\times L$ and $\ext(\th)\cap\ext(\del)
  =\{0,n\}$, where $\th={\conr{k}{r}}$
  and $\del={\conr{l}{s}}$ with $k\leq l$. Then $\th$ and $\del$ satisfy one of the
  following conditions:
  \begin{enumerate}
  \item $\alt_\del=1$ (i.e., $\del$  is the identity).
  \item $\alt_\del=2$  and $\th$ is mirrored.
  \item $\alt_\del\geq 3$ and one of the following holds:
    \begin{enumerate}
    \item None of them have a rest,
    \item $\th$ is even with rests on top and $\del$ is
      odd without rests or conversely,
    \item both are odd with rests on top,
    \item both have the same parity with rests on the bottom,
    \item $\th$ is even with rests on the bottom and $\del$ is
      odd with rests everywhere or conversely, 
      or
    \item both congruences have rests everywhere.
    \end{enumerate}
  \end{enumerate}
\end{theorem}
\begin{proof}
  We start from the second item; we obtain our conclusion by the
  contrapositive to Lemma~\ref{lem:symmetric}(\ref{item:not-mirror-trivial}).

  For item 3, since $k\leq l$ we have that $\alt_\th\geq\alt_\del\geq
  3$, and  we are under the hypothesis of
  Lemma~\ref{lem:autocruces-varios}. A straightforward but dull case
  analysis shows that some of the sub-cases applies.
\end{proof}

We underline that each case included in the list before is
possible. An enumeration of examples follows.
\begin{example}
  \begin{enumerate}
  \item $\alt_\del=1$: Trivially, an arbitrary nontrivial $\th$ would do.
  \item $\alt_\del=2$  and $\th$ is mirrored: $\con{2;2,9}\o\con{6}$
    for $n=12$. 
  \item For $\alt_\del\geq 3$:
    \begin{enumerate}
    \item None of them have a rest: this is always nontrivial. The
      reader may find this to be an easy exercise.
    \item $\th$ is even with rests on top and $\del$ is
      odd without rests or conversely:   $\con{4;4,13}\o \con{6}$ for $n=18$ and $\con{6}\o\con{7;7,22}$ for $n=30$.
    \item Both are odd with rests on top: $\con{4;4,13}\o\con{7;7}$ for $n=22$.
    \item Both have the same parity with rests on the bottom:
      $\con{4;8,17}\o\con{7;14}$ for $n=22$.
    \item $\th$ is even with rests on the bottom and $\del$ is
      odd with rests everywhere or conversely:
      $\con{7;14,29}\o\con{8;8,17,26,35}$ for $n=44$  and
      $\con{5;5,11,17,23}\o\con{7;14}$ for $n=29$. 
    \item Both congruences have rests everywhere:  $\con{7;7,15,23,31}\o\con{9;9,19,29}$ for $n=39$.
    \end{enumerate}
  \end{enumerate}
\end{example}

\renewcommand{\arraystretch}{1.5}
\begin{table}[h]
  \begin{center}
    \begin{tabular}{|c|c|}
      \hline
      \multicolumn{2}{|c|}{$\th={\conr{k}{r}}$}                                \\ \hline
      No rests         &   ${\modkk} \cup \{\<x,y\> \st x+y \modkk 0 \}$       \\ \hline
      Rests on top     &   ${\modku} \cup  \{\<x,y\> \st x+y \modku 0 \}$      \\ \hline
      Rests on bottom  &   ${\modku} \cup  \{\<x,y\> \st x+y+1 \modku 0 \}$      \\ \hline
      Rests everywhere &   ${\modkd} \cup  \{\<x,y\> \st x+y+1 \modkd 0 \}$    \\ \hline
    \end{tabular}
  \end{center}
  \caption{Simplified expression of congruences in the catalog, case
    $\alt_\th,\alt_\del\geq 3$.}
  \label{tbl:simpl-expr-congr-catal}
\end{table}

%
%
%
%
%

%
\subsection{Permutability}
\label{sec:permutability}

We now proceed to show that for any  two congruences with a
nontrivial upper bound,  their join is the composition.

It is well known that the relations $\mods$ and $\modt$ permute on any
set of integers that contains $\{0,\dots,\lcm(s,t)\}$. Moreover, this
is witnessed by the fact that every (abelian) group has a Mal'cev term
$p(x,y,z) = x-y+z$ \cite{maltsevTerm}. We will show that every pair of  congruences $\th$
and $\del$ having $\alt_\th,\alt_\del\geq 3$ and 
nontrivial $\th\o\del$ decompose essentially into the union of a modular relation
$x\mods y$ and
one of the form $x+y\modt 0$ or  $x+y+1\modt 0$, and indeed there
exist variants of the Mal'cev term showing that all of these relations
pairwise permute. 

\begin{lemma}\label{lem:permut-mod-y-mas}
  Let $s$ and $t$ be positive integers. The pairs of relations 
  \begin{enumerate}
  \item $\mods$ and $x+y\modt 0$,
  \item \label{item:4} $x+y\mods 0$ and $x+y\modt 0$,
  \item $\mods$ and $x+y+1\modt 0$,
  \item $x+y+1\mods 0$ and $x+y+1\modt 0$, and 
  \item $x+y\mods 0$ and $x+y+1\modt 0$
  \end{enumerate}
  permute on $\{0,\dots,\lcm(s,t)\}$.
\end{lemma}
\begin{proof}
  We show the first item. Assume $a\mods b$ and $b + c \modt 0$ where $a$, $b$ and $c$ are
  pairwise distinct. Let $x\doteq c+b-a$; immediately,
  we have $c  \mods c + (b -a) = x$. Also, $x +a  = c + b \modt  0$.
 
  It is clear that $-\lcm(s,t) < x < 2\lcm(s,t)$; we may obtain a solution in
  the intended range by adding $\pm \lcm(s,t)$.

  The argument for, e.g., item~\ref{item:4} uses $x\doteq
  -(a+b+c)$. The rest are very similar.
\end{proof}
\begin{corollary}\label{cor:permutability}
  Every pair of congruences appearing in
  Table~\ref{tbl:simpl-expr-congr-catal} (regardless of their join)
  permute.
\end{corollary}
\begin{proof}
  Let $\th$ and $\del$ be a pair of congruences in the table. These are of the form
  $\vartheta\cup\varpi$, where each of  $\vartheta$ and $\varpi$ are
  one of the relations $x \mods y$,  $x+y\modt 0$, or $x+y+1\modt
  0$ for suitable $s$ and $t$. It is easy
  to check that for every three relations $\phi$, $\vartheta$ and
  $\varpi$ on a set such that  $\phi$ permutes with the other two, then
  $\phi$ permutes with $\vartheta\cup\varpi$. But now we may apply
  Lemma~\ref{lem:permut-mod-y-mas} to see that the components
  $\vartheta$ and $\varpi$ of each congruence pairwise permute, and
  hence $\th$ and $\del$ permute by the previous observation.
\end{proof}
%
%
%
%
%

\begin{theorem}\label{th:permutability}
  Assume that  $\th\o\del\neq L\times L$. Then $\th\circ\del =
  \del\circ\th = \th\o\del$.
\end{theorem}
\begin{proof}
  We begin by considering the case $\ext(\th)\cap\ext(\del)
  =\{0,n\}$; we divide it into three sub-cases, whether $\alt_\del$
  equals 1 or 2, or $\alt_\del\geq 3$.
  
  If $\alt_\del=1$, then $\del=\mathrm{Id}_L$ and hence $\th\circ\del =
  \del\circ\th = \th$ regardless of any other assumption on $\th$.

  If  $\alt_\del=2$, since $\th\o\del\neq L\times L$ we infer that
  $\th$ is mirrored and hence both congruences
  permute by  Lemma~\ref{lem:symmetric}. 

  Finally, we consider the case where  $\alt_\del\geq 3$. According to
  the catalog of nontrivial joins, $\th$ and $\del$ appear in
  Table~\ref{tbl:simpl-expr-congr-catal} and hence
  Corollary~\ref{cor:permutability} implies that they permute.

  Now we wrap up the proof for the general case. Assume $x\thr y \delr
  z$. By Theorem~\ref{thm:meet} we may find $x'$, $y'$ and $z'$ between $0$ and
  $\eta_1$ such that $x\mathrel{\th\cap\del}x'$,
  $y\mathrel{\th\cap\del}y'$ and  $z\mathrel{\th\cap\del}z'$. The
  restrictions $\th'$, $\del'$ to $[0,\eta_1]$ of $\th$ and $\del$,
  respectively, are congruences on $L_{\eta_1}$ with $\ext(\th')\cap\ext(\del')
  =\{0,\eta_1\}$; therefore they permute. Find $w\in [0,\eta_1]$ such
  that $x'\delr' w \thr'  z'$. We immediately conclude that $x\delr
  w \thr  z$ by construction of $x'$ and $z'$, and we have our result.
\end{proof}
%
%
%
%
%

\subsection{Criterion for nontriviality}
\label{sec:algor-decide-non-triv}
We now proceed to give a complete criterion to decide if the join of two
congruences $\th={\conr{k}{r}}$ and $\del={\conr{l}{s}}$ is not
trivial. It can be regarded as a  converse to Theorem~\ref{thm:meet}.

As a first ingredient, we 
must enumerate the elements of $\ext(\th)\cap\ext(\del) =\{\eta_i \st
0\leq i < N\}$.
\begin{theorem}
  Assume $\th$ and $\del$ are nontrivial congruences of a line
  $\tira{L}$. Then $\th\o\del\neq L\times L$ if and only if
  \begin{enumerate}
  \item\label{item:1} for all $i$, $\eta_i$ is part of a $\th$-rest if
    and only if it is part of a $\del$-rest,
  \item\label{item:phi-cong} $\gamma\doteq{\con{\eta_1;\br\cap\bs}}$ is congruence of $\tira{L}$,
  \item the restrictions of  $\th$ and $\del$ to
    $\{0,\dots,\eta_1\}$ are a pair of  congruences appearing in the catalog,
  \item\label{item:th-compatible} the rests of  $\th$ are compatible
    with $\gamma$, and
  \item the rests of  $\del$ are compatible with $\gamma$.
  \end{enumerate}
\end{theorem}
\begin{proof}
  $(\tne)$ Assume all the conditions hold for $\th$ and $\del$ as in
  the hypothesis, and by way of contradiction assume
  $\th\o\del=L\times L$. We first solve  the problem for the
  case where $\ext(\th)\cap\ext(\del) =\{0,n\}$. By the third
  condition (the only not trivial in this case), both congruences
  appear in the catalog and then they permute by
  Corollary~\ref{cor:permutability}; hence $\del\circ\th=L\times 
  L$. 
  Since $0 \mathrel{(\del\circ\th)} k$, there
  exist  $e\in L$ such that $0 \delr e \thr k$; we have
  $e\in\ext(\th)\cap\ext(\del)$ by definition. Now consider
  $l\mathrel{(\del\circ\th)} e$; then the $e'$ satisfying  $l \delr e' \thr e$ is also a
  common extreme. Both $e$ and $e'$ are not null since $0 \mathrel{\cancel{\th}} k$, 
  and $e\neq e'$ since $0 \mathrel{\cancel{\del}} l$. This contradicts
  $\ext(\th)\cap\ext(\del) =\{0,n\}$.
  
  Next, we reduce the general case to the previous one by using the
  other conditions. It is immediate by Theorem~\ref{th:order} that
  $\gamma\subseteq\th,\del$. Now the arguments in
  Lemma~\ref{lem:frecuencias}(3) %
  will show that
  $\theta_{e_{1}}\vee\delta_{e_{1}}=(\theta\vee\delta)_{e_{1}} =
  [0,e_1]\times [0,e_1]$; but $\theta_{e_{1}}$ and $\delta_{e_{1}}$
  are congruences on $\tira{L}_{e_1}$ having
  $\ext(\th_{e_{1}})\cap\ext(\del_{e_{1}}) =\{0,e_1\}$, a
  contradiction.

  $(\ent)$ Assume that the join is not trivial. The first condition
  follows from Lemma~\ref{lem:coinciden-en-etas}. The second and the
  last two are consequences of the characterization of the order
  (Theorem~\ref{th:order}) and the 
  meet (Theorem~\ref{thm:meet}). Finally, Theorem~\ref{th:catalog}
  yields the third condition.
\end{proof}
%
%
%
%
%


%
\subsection{Computing the join}
\label{sec:computing-join}
In this section we will compute the step of the join of two
congruences; that is, given $\th,\del \in \Con \tira{L}$, we will
calculate the cardinality of $\tira{L}/\th\o\del$.

In the following lemmas before 
Theorem~\ref{thm:caracterizacion sup e inf}, we work under the
assumption that 
$\ext(\th)\cap \ext(\del) =\{0,n\}$ and $\th\o\del\neq L\times L$. By
the results of Section~\ref{sec:permutability}, this
implies that $\th$ and $\del$ permute. 

\begin{lemma}\label{lem:extremos}
  Assume that $\th$ or  $\del$ has at least one rest. 
  Then  for all $x\in\ext(\th\o\del)$, either 
  \begin{enumerate}
  \item $x\in\ext(\th)\cup\ext(\del)$ or
  \item $x$ is at a distance strictly less than 1 from a crossing.
  \end{enumerate}
  Moreover, if $i$ ($h$) denotes the number of crossings at (half-)
  integral coordinates, we have
  \[\card{\ext(\th\o\del)} = \card{\ext(\th)} +  \card{\ext(\del)} +
  2i + 4h -2.\]
\end{lemma}
\begin{proof}
  It is clear that
  $\ext(\th)\cup\ext(\del)\subseteq\ext(\th\o\del)$, and that any
  $x\in L$ at a distance smaller than 1 from a crossing belongs to
  $\ext(\th\o\del)$. So it remains to be checked that every $x\in
  \ext(\th\o\del) \setminus (\ext(\th)\cup\ext(\del))$ is at distance
  less than one from a crossing. 
  Take such an $x$.  Since $x\notin \ext(\th)\cup\ext(\del)$, $x$ is
  in the interior of the trajectory diagram. There are two cases,
  depending on whether $x$ is part of a rest, or not. 

  Assume that $x$ is not part of a rest of $\th\o\del$. Hence
  $x-1\mathrel{\th\o\del}x+1$. By permutability, there must exist
  $y,y'$ such that  $x-1\thr y \delr x+1$ and $x-1\delr y'\thr
  x+1$. If $y' = y\pm 2$, the segment joining $y$ and $y'$ crosses at
  $x$ the  one determined by $x-1$ and $x+1$, and we are
  done. Otherwise, if $x$ does not belong to the line segment $\overline{yy'}$,
  each of them must lie in parallel lines at horizontal distance 2; see Figure~\ref{fig:horiz-dist}. It can be easily seen by considering the bounces of
  these three fragments of the trajectory that either $\th\o\del$ is
  trivial or none of them  has any rest. To finish this case, observe
  that this kind of crossing contributes with 2 extremes of
  $\th\o\del$ ($x$ and $y\pm1$).
  
  \definecolor{eqdfdf}{rgb}{0.88,0.87,0.87}
  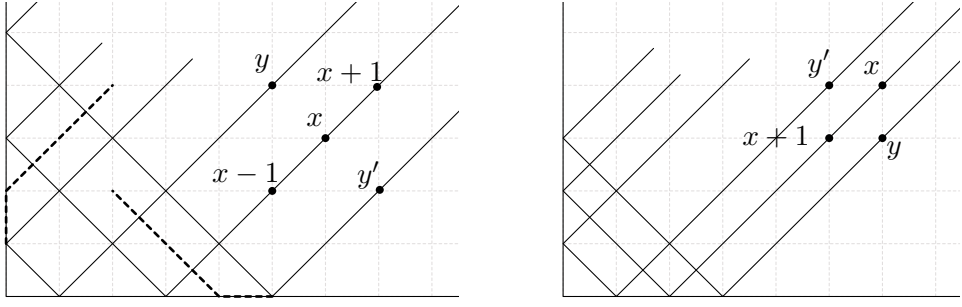
\begin{figure}
    \begin{center}
    \begin{tabular}{ccc}
      \begin{tikzpicture}[line cap=round,line join=round,>=triangle 45,x=0.7cm,y=0.7cm]
        \draw [color=eqdfdf,dash pattern=on 1pt off 1pt, xstep=0.7cm,ystep=0.7cm] (2,-3) grid (10.5,2.57);
        \clip(1.63,-3.32) rectangle (10.5,2.57);
        \draw (10,4)-- (3,-3)-- (2,-2)-- (5.5,1.5);
        \draw (14,4)-- (7,-3)-- (2,2)-- (3,3);
        \draw (12,4)-- (5,-3)-- (2,0)-- (3.8,1.8);
        \draw [line width=1pt,dash pattern=on 2pt off 2pt] (7,-3)-- (6,-3)-- (4,-1);
        \draw [line width=1pt,dash pattern=on 2pt off 2pt] (2,-2)-- (2,-1)-- (4,1);
        \draw [line width=0.4pt] (16,4)-- (2,4)-- (2,-3)-- (16,-3);
        \fill [color=black] (7,-1) circle (1.5pt);
        \draw[color=black] (6.5,-0.68) node {$x -1$};
        \fill [color=black] (8,0) circle (1.5pt);
        \draw[color=black] (7.8,0.34) node {$x$};
        \fill [color=black] (8.97,0.97) circle (1.5pt);
        \draw[color=black] (8.46,1.24) node {$x+1$};
        \fill [color=black] (7,1) circle (1.5pt);
        \draw[color=black] (6.81,1.42) node {$y$};
        \fill [color=black] (9.02,-0.98) circle (1.5pt);
        \draw[color=black] (8.81,-0.7) node {$y'$};
      \end{tikzpicture}
      & \ &
      \definecolor{eqdfdf}{rgb}{0.88,0.87,0.87}
      \definecolor{wwwwww}{rgb}{0.6,0.6,0.6}
      \begin{tikzpicture}[line cap=round,line join=round,>=triangle 45,x=0.7cm,y=0.7cm]
        \draw [color=eqdfdf,dash pattern=on 1pt off 1pt, xstep=0.7cm,ystep=0.7cm] (3,-3) grid (10.45,2.57);
        \clip(2.62,-3.32) rectangle (10.45,2.57);
        \draw (11,4)-- (4,-3)-- (3,-2)-- (6.5,1.5);
        \draw (12,4)-- (5,-3)-- (3,-1)-- (5.2,1.2);
        \draw (13,4)-- (6,-3)-- (3,0)-- (4.7,1.7);
        \draw [line width=0.4pt] (17,4)-- (3,4)-- (3,-3)-- (17,-3);
          \fill [color=black] (9,1) circle (1.5pt);
          \draw[color=black] (8.8,1.34) node {$x$};
          \fill [color=black] (8,0) circle (1.5pt);
          \draw[color=black] (7,0.03) node {$x+1$};
          \fill [color=black] (8,1) circle (1.5pt);
          \draw[color=black] (7.81,1.41) node {$y'$};
          \fill [color=black] (9,0) circle (1.5pt);
          \draw[color=black] (9.22,-0.24) node {$y$};
      \end{tikzpicture}
    \end{tabular}
    \end{center}
    \caption{Fragments at horizontal distance 2 and 1, respectively.}\label{fig:horiz-dist}
  \end{figure}
  The case for $x$ being part of a rest of  $\th\o\del$  is very
  similar. For instance, if $x-1$ is a rest, then $x-1
  \mathrel{\th\o\del} x$,  and by permutability, there must exist
  $y,y'$ such that  $x-1\thr y \delr x$ and $x-1\delr y'\thr
  x$. It is easily seen that for nontrivial  $\th\o\del$  we must
  have  $y' = y\pm 1$, and then this kind of crossing contributes with
  4 in the count of  $\ext(\th\o\del)$ (namely, $x-1$, $x'$, $y$, $y'$).
\end{proof}
For the purpose of the next proofs, let $r_\th$ denote the number of rests of
$\th$ and let $c_{\th,\del}$ denote the number of crossings in the trajectory
diagram of $\th\o\del$. From Equation~(\ref{eq:1}) we have that 
\[
r_\th = \card{\ext(\th)}-\alt_{\th}-1.
\] 
\begin{corollary}\label{cor:bounces+crossings}
  If  any of $\th$, $\del$ has at least one rest,
  we have  $\alt_{\th\o\del}= \alt_\th+\alt_\del-1 + 2\cdot c_{\th,\del}$.
\end{corollary}
\begin{proof}
  By the proof
  of Lemma~\ref{lem:extremos}, rests of $\th\o\del$ correspond to
  rests of either $\th$ or $\del$, or to crossings at half-integral
  coordinates, the latter contributing with two rests each. We keep
  the notation of the previous Lemma ($h,i$ for the number of
  crossings at half-integral, integral coordinates, resp.). So we have
  \begin{align*}
  \alt_{\th\o\del} & =  \card{\ext(\th\o\del)} -1 - r_{\th\o\del} \\
  \alt_{\th\o\del} & =  \card{\ext(\th\o\del)} -1 - \bigl( r_{\th} +
  r_{\del} + 2h\bigr) \\
  & =  \card{\ext(\th\o\del)} - 1  - \bigl( \card{\ext(\th)}-\alt_\th
  -1 +  \card{\ext(\del)} -\alt_\del -1 + 2h\bigr)\\
  & =  \card{\ext(\th)} +  \card{\ext(\del)} +
  2i + 4h -2 -1 - \card{\ext(\th)}+\alt_\th + 1 -
  \card{\ext(\del)} +\alt_\del +1 - 2h\\
  & =  
  2i+ 2h  -1 + \alt_\th +\alt_\del \\
  & = \alt_\th +\alt_\del -1 + 2\cdot c_{\th,\del}.
  \end{align*}
\end{proof}
\begin{lemma}
  \label{lem:paridad-alternancia}
  One of  $\alt_\th$, $\alt_\del$ must be odd.
\end{lemma}
\begin{proof}
  We show that  if both  $\alt_\th$ and $\alt_\del$ are even, 
  $\ext(\th)\cap \ext(\del) \neq\{0,n\}$.
  
  First observe that if both frequencies are even, then neither $k\thr
  n$ nor $l \delr n$ hold. For $\th={\conr{k}{r}}$, this happens because 
  \[n - \displ{r}(n) = n - \card{\br} = \alt_\th\cdot k \modkk 0 = 0 -
  \displ{r}(0),\]
  and hence $0 \thr n$. But by definition of $k$, $k \thr 0$ is not
  possible; hence $k \nthr n$. The same works for $\del$.
  
  We proceed by considering some cases. If $0 \mathrel{(\th\o\del)} k$,
  by permutability there must exist $x$ such that $0\delr x \thr k$;
  this $x$  belongs to  $\ext(\th)\cap \ext(\del)\setminus
  \{0,n\}$, because $k\nthr 0$ and $k\nthr n$. If $0
  \mathrel{(\th\o\del)} l$ we arrive at the same consequence \emph{mutatis mutandis}.

  Otherwise, assume $(0,k), (0,l)\notin \th\o\del$. Since
  $l\in\ext(\del)\subseteq\ext(\th\o\del)$ (Lemma~\ref{lem:aux-meet}.1), we have $l \mathrel{(\th\o\del)}
  k$. We may then find some $x$ such
  that  $l\delr x \thr k$ holds. But again $x\in \ext(\th)\cap \ext(\del)\setminus
  \{0,n\}$ and we are done.
\end{proof}
\begin{lemma}\label{lem:overlap}
  There exist no distinct $x,y,y'\in L$ such that $y$ and $y'$ are
  consecutive, $x\mathrel{\th\cap\del} y$ and $x+1\mathrel{\th\cap\del} y'$.
\end{lemma}
\begin{proof}
  By way of contradiction, assume there exist such $x,y,y'$ with $x$
  minimal with this property. 
  
  We first check $x\neq 0$; otherwise, since
  $y\mathrel{\th\cap\del} x$, then
  $y\in\ext(\th)\cap\ext(\del)=\{0,n\}$. Therefore, $y=n$, and we have
  \[
  0\modkk n-|\br| = k\cdot \alt_\th,
  \]
  and hence $2 \mid \alt_\th$. Similarly, $2\mid \alt_\del$, but this
  contradicts Lemma~\ref{lem:paridad-alternancia}.

  We will show that $x$ must be a common extreme of both $\th$ and
  $\del$. 

  We first consider the case where $y'=y+1$, and perform an analysis
  of the $\th$-relationships. Since $x\thr y$,
  $x+1\thr y+1$ we have four options, according to the signs below:
  \begin{align}
    x-\displ{r}(x)&\modkk \pm(y -\displ{r}(y))\label{eq:x-con-y} \\
    x+1-\displ{r}(x+1)&\modkk \pm(y+1
    -\displ{r}(y+1))\label{eq:x+1-con-y+1},
  \end{align}
  which we denote by $(++)$, $(+-)$, $(-+)$, and $(--)$. By minimality
  of $x$, we have
  \begin{equation}\label{eq:minimal}
    x-1-\displ{r}(x-1)\not\modkk y-1 -\displ{r}(y-1).
  \end{equation}
  
  Cases $(++)$, $(+-)$, and $(-+)$ can be treated uniformly; in each
  of them, there is a $d\in\{0,1\}$ such that
  \begin{equation}
    x+d-\displ{r}(x+d)\modkk y+d
    -\displ{r}(y+d).\label{eq:x+d-con-y+d}
  \end{equation}
  Add 1 to each side of  (\ref{eq:minimal}) and  subtract from
  (\ref{eq:x+d-con-y+d}):
  \begin{equation*}
    d+\displ{r}(x-1)-\displ{r}(x+d)\not\modkk d + \displ{r}(y-1)  -\displ{r}(y+d)
  \end{equation*}
  and then
  \begin{equation*}
    \displ{r}(x-1)-\displ{r}(x+d)\not\modkk \displ{r}(y-1)  -\displ{r}(y+d).
  \end{equation*}
  Since $\th$ is not trivial, only one of $x-1,x$ can be a rest; the
  same happens with $y$. And then each side of the inequality is
  either $1$ or $0$. So at least one of $x,y$ is part of  rest, but
  then both are, because $x \thr y$. Hence $x\in\ext(\th)$.
  
  For case $(--)$, we subtract the corresponding version of
  (\ref{eq:x+1-con-y+1}) from (\ref{eq:x-con-y}):
  \[
    -\displ{r}(x)-1+\displ{r}(x+1)\modkk \displ{r}(y) +1 -\displ{r}(y+1)
  \]
  and then
  \[
    \bigl(\displ{r}(x+1)-\displ{r}(x)\bigr)+\bigl(\displ{r}(y+1)-
    \displ{r}(y)\bigr) \modkk 2.
  \]
  Since both differences are $1$ or $0$, the only way that this holds
  is that both $x,y$ are rests (hence extremes) or $k=1$ (in which
  case \emph{every} element is a extreme). 
  
  We conclude that in each of the four cases, $x\in\ext(\th)$. The
  same analysis applies to $\del$-relationships, hence 
  $x\in\ext(\th)\cap\ext(\del)$.
  
  Now consider the case where $y'=y-1$. As
  before,  the $\th$-relationships are
  \begin{align}
    x-\displ{r}(x)&\modkk \pm(y -\displ{r}(y))\label{eq:x-con-y-nueva} \\
    x+1-\displ{r}(x+1)&\modkk \pm(y-1
    -\displ{r}(y-1))\label{eq:x+1-con-y-1},
  \end{align}
  and 
  \begin{align}
    x-1-\displ{r}(x-1)&\not\modkk y+1
    -\displ{r}(y+1)\label{eq:minimal-nueva}\\
    x-1-\displ{r}(x-1)&\not\modkk -y-1
    +\displ{r}(y+1)\label{eq:minimal-nueva-menos}    
  \end{align}
  by minimality of $x$. 
  
  The corresponding  $(++)$ case is similar to $(--)$ above: By
  subtracting (\ref{eq:x-con-y-nueva}) from (\ref{eq:x+1-con-y-1})  we
  conclude that either
  $x,y$ are parts of rests or $k=1$, and then $x\in\ext(\th)$.

  For the remaining cases, observe that if some of
  $x,y,x-1,y-1$ is a rest, we have $x\in\ext(\th)$ as before.
  Hence we will work under this assumptions:
  \begin{equation}\label{eq:no-rests}
    \begin{split}
      \displ{r}(x-1)&= \displ{r}(x)= \displ{r}(x+1)\\
      \displ{r}(y-1)&=\displ{r}(y) = \displ{r}(y+1).
    \end{split}
  \end{equation}

  Cases $(-+)$ and $(--)$ are impossible: Subtract
  (\ref{eq:minimal-nueva-menos}) from~(\ref{eq:x-con-y-nueva}), 
  \[
         x-\displ{r}(x) -x+1+\displ{r}(x-1)\not\modkk  -y
         +\displ{r}(y) + y+1 -\displ{r}(y+1),
  \]
  and  applying (\ref{eq:no-rests}) we get $1\not\modkk 1$, a
  contradiction.
 
  Case $(+-)$:  Using (\ref{eq:no-rests}) we may add
  (\ref{eq:x-con-y-nueva}) and   (\ref{eq:x+1-con-y-1}) and  obtain:
  \[
  2(x- \displ{r}(x)) \modkk 0,
  \]
  and then $x- \displ{r}(x) \modkk k$ or $0$, and hence
  $x\in\ext(\th)$.

  As before, the same analysis applies to $\del$, hence
  $x\in\ext(\th)\cap\ext(\del)$.
\end{proof}
\begin{lemma}\label{lem:cruces}
  For all  $\th$ and  $\del$,  $ c_{\th,\del} = \tfrac{1}{2}(\alt_\th-1)(\alt_\del
  -1)$.
\end{lemma}
\begin{proof}
  Given two consecutive elements $x<x'$ of $\ext(\del)$ such that $x$
  is not a rest (hence  
  $x'-x=l$), call \emph{trip} the fragment of the line $\tira{L}$
  going from $x$ to $x'$. Thus a trip corresponds to a polygonal piece
  of the trajectory of $\th\o\del$ going across the whole diagram from left to right
  or viceversa. 

  The crucial observation is that for nontrivial
  $\th\o\del$, Lemma~\ref{lem:overlap} implies that trips  do
  not overlap and therefore intersect each other at finitely many
  points. This allows us to use a counting argument to calculate the
  number of crossings. Given $x,x'$ as above and consecutive
  $y,y'\in\ext(\th)$ such that
  \begin{itemize}
  \item $y<y'$,
  \item none of them  is between $x$ and  $x'$, and 
  \item  $y$ not a rest (hence $y'-y=k$), 
  \end{itemize}
  the trip going from $x$ to
  $x'$ crosses the polygonal piece  $\overline{yy'}$  of the
  trajectory exactly once. Hence
  $\overline{yy'}$  is crossed by all the trips that do not contain
  $y$ nor $y'$. We may consider two types of pieces
  $\overline{yy'}$: Those such that $y,y'$ belong to different trips
  (i.e., there is $y<x<y'$ with $x\in\ext(\del)$) or those such that
  $y,y'$ belong to the same trip (see Figure~\ref{fig:trip}). 
  \begin{figure}
    \begin{center}
      \definecolor{tttttt}{rgb}{0.7,0.7,0.7}
      \definecolor{wwwwww}{rgb}{0.6,0.6,0.6}
      \begin{tikzpicture}[line cap=round,line join=round,>=triangle 45,x=0.5cm,y=0.5cm]
        \draw [line width=0.4pt,color=wwwwww] (8,14)-- (8,5);
        \draw [line width=0.4pt,color=wwwwww] (8,5)-- (28,5);
        \draw [line width=0.4pt,color=wwwwww] (28,5)-- (28,14);
        \draw [line width=0.4pt,color=wwwwww] (28,14)-- (8,14);
        \draw [line width=0.4pt,color=tttttt] (8,5)-- (17,14)-- (18,14)-- (27,5);
        \draw [line width=0.4pt,color=tttttt] (11,14)-- (8,11)-- (14,5)-- (23,14);
        \draw [line width=0.4pt,color=tttttt] (11,14)-- (12,14)-- (21,5);
        \draw [line width=0.4pt,color=tttttt] (27,8)-- (21,14)-- (20,14)-- (11,5)-- (8,8)-- (14,14)-- (15,14)-- (24,5);
        \draw [line width=0.4pt,color=tttttt] (21,5)-- (27,11);
        \draw [line width=0.4pt,color=tttttt] (24,5)-- (27,8);
        \draw [line width=0.4pt,color=tttttt] (27,8)-- (28,7)-- (28,6)-- (27,5);
        \draw [line width=0.4pt,color=tttttt] (27,8)-- (28,9)-- (28,10)-- (26,12);
        \draw [line width=0.4pt,color=tttttt] (26,12)-- (24,14)-- (23,14);
        \draw [line width=0.4pt,color=tttttt] (27,11)-- (28,12)--
        (28,13)-- (27,14)-- (26,14)-- (17,5)-- (8,14);
        \draw [line width=0.7pt] (28,7)-- (21,14)-- (20,14)-- (11,5)-- (8,8);
        \draw [line width=1pt] (27,5)-- (28,6)-- (28,7)-- (21,14);
        \draw [line width=1pt] (20,14)-- (11,5);
        \draw (18.74,13.26)-- (19.29,12.71);
        \draw (17.32,11.68)-- (17.78,11.22);
        \draw (15.74,10.26)-- (16.32,9.68);
        \draw (13.71,8.29)-- (14.31,7.69);
        \draw (12.26,6.74)-- (12.78,6.22);
        \draw (22.23,13.23)-- (21.76,12.76);
        \draw (23.67,11.67)-- (23.25,11.25);
        \draw (25.71,9.71)-- (25.26,9.26);
        \draw (27.3,8.3)-- (26.77,7.77);
        \fill [color=black] (21,14) circle (1.5pt);
        \draw[color=black] (21.02,14.54) node {$y_1'$};
        \fill [color=black] (20,14) circle (1.5pt);
        \draw[color=black] (19.96,14.52) node {$y_2$};
        \fill [color=black] (11,5) circle (1.5pt);
        \draw[color=black] (11.04,4.54) node {$y_2'$};
        \fill [color=black] (8,8) circle (1.5pt);
        \draw[color=black] (7.38,8.16) node {$x'$};
        \fill [color=black] (27,5) circle (1.5pt);
        \draw[color=black] (27.18,4.58) node {$y_1$};
        \fill [color=black] (28,7) circle (1.5pt);
        \draw[color=black] (28.32,7.28) node {$x$};
      \end{tikzpicture}
    \end{center}
    \caption{A trip $\overline{xx'}$ and pieces
      $\overline{y_1y_1'}$,  $\overline{y_2y_2'}$ of resp.\ first and second
        type.} \label{fig:trip}
  \end{figure}
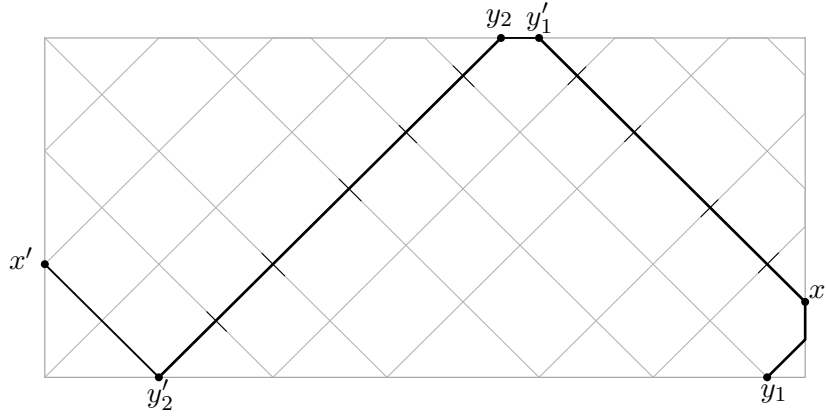
  There are  $\alt_\del-1$ pieces of the first type, and each
  of them is crossed $\alt_\del-2$ times. 

  On the other hand, the total number pieces is equal to $\alt_\th$,
  hence the number pieces of the second type is $\alt_\th
  -(\alt_\del -1)$, and  each  of them is crossed $\alt_\del-1$ times. 

  In this calculation, we have counted each crossing twice, so we
  finally obtain the total number of crossings as
  \[\frac{1}{2}\bigl[(\alt_\del-1)(\alt_\del-2) + \bigl(\alt_\th-(\alt_\del
  -1)\bigr)(\alt_\del-1)\bigr] = \frac{1}{2}(\alt_\th-1)(\alt_\del-1).\]
  \end{proof}

\begin{theorem}\label{th:freq-when-e1=n}
  Assume $\th\o\del\neq L\times L$ and   $\ext(\th)\cap \ext(\del) =\{0,n\}$. 
  Then $\alt_{\th\o\del} = \alt_\th\cdot\alt_\del$.
\end{theorem}
\begin{proof}
  Assume first that at least one of the congruences has a rest. By
  Corollary~\ref{cor:bounces+crossings}, we obtain
  \[ 
  \alt_{\th\o\del} = \alt_\th+\alt_\del -1 + 2 c_{\th,\del}.
  \]
  Using Lemma~\ref{lem:cruces},
  \[ \alt_\th+\alt_\del-1 + (\alt_\th-1)(\alt_\del-1) =
  \alt_\th\cdot\alt_\del,\]
  and we are done. 

  Now we turn to the case without rests. That is, 
  assume $\th={\con{k}}$ and $\del={\con{l}}$. Note that $k$ and
  $l$ divide $n$ by Theorem~\ref{th:charact-congruences}. Also observe that
  $\ext(\th) =  \{t\cdot k : t \in [0,\tfrac{n}{k}]\}$ and analogously for $\del$.
  In particular, $\lcm(k,l)\in\ext(\th)\cap\ext(\del)$ and hence
  $n=\lcm(k,l)$. So, 
  \[\alt_\th = \frac{n}{k} =\frac{\lcm(k,l)}{k}, \quad \text{ and similarly
  } \quad \alt_\del = \frac{\lcm(k,l)}{l}.\]
   It is easy to see  now that $\gcd(\alt_\th,\alt_\del) =1$. On the other hand, by
  Lemma~\ref{lem:frecuencias}.(2), we know that both $\alt_\th$ and
  $\alt_\del$ divide $\alt_{\th\o\del}$, therefore
  $\alt_\th\cdot\alt_\del \mid \alt_{\th\o\del}$. Finally, $\con{\gcd(k,l)}$ is
  an upper bound of $\{\th,\del\}$ in  $\Con\tira{L}$, having
  frequency $\alt_\th\cdot\alt_\del$; therefore we obtain
  $\alt_{\th\o\del} = \alt_\th\cdot\alt_\del$.    
\end{proof}

\begin{corollary}
  Assume that for $\th={\conr{k}{r}}$, $\del={\conr{l}{s}}$,  $\th\o\del\neq L\times
  L$ and $\ext(\th)\cap \ext(\del) =\{0,n\}$. Then the cardinality of
  $\tira{L}/\th\o\del$ is  
  \[\left\lfloor\frac{nkl}{(n-\card{\br})(n-\card{\bs})}\right\rfloor.\]
\end{corollary}
\begin{proof}
  Immediate since $\alt_\th=\tfrac{n-\card{\br}}{k}$ (analogously with
  $\alt_\del$) and the step of
    the congruence $\th\o\del$ is equal to the integral division of $n$
    by $\alt_{\th\o\del}$.
\end{proof}

We need now one further result in order to  lift the assumption of
$\ext(\del)\cap \ext(\del) =\{0,n\}$.  Given congruences $\th,\del$ in
the catalog of nontrivial joins, it is 
useful to be able to recover the length $n$ of the line given the
frequency $\alt_\del$ and the step $l$ of $\del$.  In
Table~\ref{th:values-n-terms-step} we gather the values of
$n=n(l,\alt_\del)$ in terms of  the corresponding parameters $l$ and $\alt_\del$.

\begin{table}[h]
  \begin{center}
    \begin{tabular}{c|c|c|}
      \cline{2-3}
        & \multicolumn{2}{c|}{$n = n(l,\alt_\del)$}     \\ \cline{2-3}
      & \multicolumn{1}{c|}{$\alt_\del$ even} & 
      \multicolumn{1}{c|}{ $\alt_\del$ odd} \\ \hline
     \multicolumn{1}{|c|}{ No rests}         & \multicolumn{2}{c|}{$l\,\alt_\del$}             \\ \hline
      \multicolumn{1}{|c|}{Rests on top}     & $(2l+1)\tfrac{\alt_\del}{2} $    &
      $\frac{1}{2}[(2l+1)\alt_\del-1]$      \\ \hline
      \multicolumn{1}{|c|}{Rests on bottom}  &
      $(2l+1)\frac{\alt_\del}{2} -1 $ &
      $\frac{1}{2}[(2l+1)\alt_\del-1]$     \\ \hline 
      \multicolumn{1}{|c|}{Rests everywhere} & \multicolumn{2}{c|}{$(l+1)\alt_\del-1$}       \\ \hline
    \end{tabular}
  \end{center}
  \caption{Values of $n$ in terms of step and frequency (case
    $\alt_\del\geq 3$).}\label{th:values-n-terms-step}
\end{table}

Note that actually, for $0\leq i\leq\alt_\del$, the values $n(l,i)$ belong
to $\ext(\del)$ (using the formula according to the parity of
$i$). These values will appear in the proof of the following result.

\begin{lemma}\label{lem:barney-theory}
  Assume that $\ext(\th)\cap \ext(\del) =\{0,n\}$ and $\th\o\del\neq
  L\times L$. Then $\gcd(\alt_\th,\alt_\del)=1$. 
\end{lemma}
\begin{proof}
  The main argument will be to show that if $\alt_\th$ and $\alt_\del$
  have a common divisor, then $\th$ and $\del$ have a common non null
  extreme before $n$.  In the following we will make silent use of
  Lemma~\ref{lem:paridad-alternancia}.

  If $\alt_\del=1$, the result follows. If $\alt_\del=2$, by the
  assumption $\ext(\th)\cap \ext(\del) =\{0,n\}$ we conclude that
  $2\nmid \alt_\th$, and this case is also ready.

  Next, assume that $\alt_\th,\alt_\del \geq 3$. We have to analyze all cases
  of the catalog of nontrivial joins. We first consider the case in
  which $\th$ has no rests. 
  \begin{enumerate}[1.]
  \item\label{item:sin_sin} \emph{$\del$ has no rests.} 
    Every common, non null extreme of $\th$ and $\del$  is of the form
    $x = k\,a = l\,b$
    for some   $0<a\leq \alt_\del$, $0<b \leq \alt_\th$. Hence, 
    if $d>0$ is a common divisor of $\alt_\th$ and $\alt_\del$, we
    may choose $a$ and $b$ such that  $\alt_\th = d \cdot a$ and $\alt_\del = d \cdot b$, hence $k \cdot a = l\cdot 
    b $ is a common non null extreme of $\th$ and $\del$, hence it
    must be $n$. We conclude $d=1$.
  \item \emph{$\del$ has rests on top.} There are two sub-cases,
    according to whether $2 \mid \alt_\del$ or not. If $\alt_\del$ is
    even, $\alt_\th$ must be odd. We have 
    $k\,\alt_\th = (2l+1)\tfrac{\alt_\del}{2},$
    and thus $ 2 k\,\alt_\th = (2l+1)\alt_\del$. We may use the same
    reasoning in item~\ref{item:sin_sin} to conclude that
    $\gcd(\alt_\th,\alt_\del) =1$.

  \end{enumerate}
  
  There remain four more cases, when both congruences have at least
  one rest.
  
  \begin{enumerate}[{\it a.}]
  \item \emph{Both congruences have rests on top.} In this case both
    congruences must have odd frequencies. We
    have
    $n = \frac{1}{2}[(2k+1)\alt_\th-1]= \frac{1}{2}[(2l+1)\alt_\del-1]$,
    hence $(2k+1)\alt_\th  = (2l+1)\alt_\del$. We may then reason as
    in item~\ref{item:sin_sin}.
  \item \emph{Both congruences have rests on bottom.} As in the
    previous case, both frequencies are odd. Since the values of $n$
    are exactly the same as in the previous item, we have immediately
    the conclusion.
  \item  \emph{$\th$ has rests on bottom, and $\del$ has rests
    everywhere.} In this case, $\alt_\th$ must be even. We have
    $(2k+1)\frac{\alt_\th}{2} -1  = (l+1)\alt_\del-1$.
    Hence we obtain $(2k+1)\alt_\th = 2(l+1)\alt_\del$, and we may
    reason as in item~\ref{item:sin_sin}.
  \item \emph{Both congruences have rests everywhere.} Here we equate
    $(k+1)\alt_\th-1 = (l+1)\alt_\del-1$
    and by adding 1 we are led to the same case as item~\ref{item:sin_sin}.\qedhere
  \end{enumerate}    
\end{proof}

Next we obtain a simple formula for the frequency of
${\th\o\del}$ in terms of the original frequencies, thereby enabling us to
calculate the step of a nontrivial join in general. 
\begin{thm}
\label{thm:caracterizacion sup e inf}Let $\theta$ and $\delta$
be congruences of $\mathbf{L}_{n}$ with nontrivial join. Then, $\alt_{\theta\vee\delta}=\lcm\left(\alt_{\theta},\alt_{\delta}\right)$
and $\alt_{\theta\wedge\delta}=\gcd\left(\alt_{\theta},\alt_{\delta}\right)$.\end{thm}
\begin{proof}
Let $e_{1}$ be the first positive element in $\ext\left(\theta\right)\cap\ext\left(\delta\right)$.
Observe that $\theta_{e_{1}}$ and $\delta_{e_{1}}$ are congruences
of $\mathbf{L}_{e_{1}}$, and $\theta_{e_{1}}\vee\delta_{e_{1}}$
is nontrivial (because otherwise $\left\langle 0,1\right\rangle \in\theta_{e_{1}}\vee\delta_{e_{1}}\subseteq\theta\vee\delta$).
So Theorem~\ref{th:freq-when-e1=n} yields 
\[
\alt_{\theta_{e_{1}}\vee\delta_{e_{1}}}=\alt_{\theta_{e_{1}}}\alt_{\delta_{e_{1}}},
\]
and Lemma~\ref{lem:barney-theory} says that 
\[
\gcd\left(\alt_{\theta_{e_{1}}},\alt_{\delta_{e_{1}}}\right)=1.
\]
Recall from Lemma~\ref{lem:frecuencias} that 
\begin{align*}
\alt_{\theta\vee\delta} & =\alt_{\theta_{e_{1}}\vee\delta_{e_{1}}}\alt_{\theta\wedge\delta}\\
\alt_{\theta} & =\alt_{\theta_{e_{1}}}\alt_{\theta\wedge\delta}\\
\alt_{\delta} & =\alt_{\delta_{e_{1}}}\alt_{\theta\wedge\delta}.
\end{align*}
Now both desired formulas follow easily: 
\begin{align*}
\lcm\left(\alt_{\theta},\alt_{\delta}\right) & =\lcm\left(\alt_{\theta_{e_{1}}}\alt_{\theta\wedge\delta},\alt_{\delta_{e_{1}}}\alt_{\theta\wedge\delta}\right)\\
 & =\lcm\left(\alt_{\theta_{e_{1}}},\alt_{\delta_{e_{1}}}\right)\alt_{\theta\wedge\delta}\\
 & =\alt_{\theta_{e_{1}}}\alt_{\delta_{e_{1}}}\alt_{\theta\wedge\delta}\\
 & =\alt_{\theta_{e_{1}}\vee\delta_{e_{1}}}\alt_{\theta\wedge\delta}\\
 & =\alt_{\theta\vee\delta},
\end{align*}
and
\begin{align*}
\gcd\left(\alt_{\theta},\alt_{\delta}\right) & =\gcd\left(\alt_{\theta_{e_{1}}}\alt_{\theta\wedge\delta},\alt_{\delta_{e_{1}}}\alt_{\theta\wedge\delta}\right)\\
 & =\gcd\left(\alt_{\theta_{e_{1}}},\alt_{\delta_{e_{1}}}\right)\alt_{\theta\wedge\delta}\\
 & =\alt_{\theta\wedge\delta}.
\end{align*}
\end{proof}
\begin{corollary}
  Assume that $\th={\conr{k}{r}}$, $\del={\conr{l}{s}}$, and  $\th\o\del\neq L\times
  L$. Then the cardinality of $\tira{L}/\th\o\del$ is the integral
  division of $n$ by
  \[\lcm\left(\frac{n-\card{\br}}{k},\frac{n-\card{\bs}}{l}\right).\]
\end{corollary}
We finish by making some observations concerning the class of
lattices of the form $\Con \tira{L}_n$. It is not difficult to see that for
prime $p$, the non distributive lattice $M_{p-2}$ embeds into $\Con
\tira{L}_p$. Also, $\Con \tira{L}_9$ has a sublattice isomorphic to $N_5$: It is
generated by the congruences $\con{4;4}$, $\con{2;4}$, and
$\con{1}$. As a consequence, the class of congruence lattices of
lines does not satisfy the modular law. But
on the other hand, every proper ideal of $\Con \tira{L}_n$ must be
modular indeed, since it is a lattice of permuting relations. Even more is
true:
\begin{theorem}
Let $\rho\neq L\times L$ be a congruence of $\mathbf{L}_{n}$, and
let
\[(\rho]\doteq\left\{ \vartheta\subseteq\rho\mid\vartheta\mbox{ is a
    congruence of }\tira{L}_{n}\right\}.\]
Then the map 
\[
\vartheta\mapsto\alt_{\vartheta}
\]
is a lattice embedding from $\left\langle (\rho],\wedge,\vee,\mathrm{Id}_{L_{n}},\rho\right\rangle $
into $\bigl\langle \{ \mbox{positive divisors of }\alt_{\rho}\} ,\gcd,\lcm,1,\alt_{\rho}\bigr\rangle $.
In particular $\left\langle (\rho],\wedge,\vee\right\rangle $ is
distributive.
\end{theorem}
\begin{proof}
In view of Theorem \ref{thm:caracterizacion sup e inf} it only remains
to be shown that the map is injective. Assume $\alt_{\theta}=\alt_{\delta}$.
Then $\theta$ and $\delta$ have the same step; say $\theta={\conr{k}{r}}$
and $\delta={\conr{k}{s}}$. For the sake of contradiction suppose $\bar{r}\neq\bar{s}$,
and let $i$ be the first index such that $r_{i}\neq s_{i}$. If $r_{i}<s_{i}$,
then $r_{i}$ is a common extreme, and Lemma \ref{lem:coinciden-en-etas}
says that $r_{i}\in\bar{s}$, a contradiction. Of course the case
$s_{i}<r_{i}$ is symmetric, and thus $\bar{r}=\bar{s}$.\end{proof}

%
%
%
%
%

%
\section{Conclusions \& Further Work}
\label{sec:conclusions--further-work}%


As mentioned in the introduction, the starting point for this article
is the study of equationally definable functions in the variety of
modal algebras. A critical step in this study is understanding
the lattice of subalgebras of a modal algebra, which is dually isomorphic
to the lattice of congruences of the dual frame of the algebra. We
were able to obtain a very good description of this lattice for line
frames. Congruences on line frames are represented as a tuple of
integer parameters. Employing this representation we were able to
determine their order relation and describe how their meets and joins
are computed, showing that the latter coincide with the relational
composition whenever the result is not trivial. The study of the lattice
operations on congruences required extensive combinatorial reasoning,
which we approached geometrically by using trajectory diagrams --- the
essential tool behind the results of the paper. We obtained explicit
formulas for the cardinality of the quotient of a line by the join
of two congruences, by using the said parameters. Finally, we provided
a description of the general structural of congruence lattices of
line frames by proving that they are composed of lattices of divisors
of integer numbers, with a new top element attached (in particular,
these lattices are not modular in general, but every proper decreasing
subset is distributive).

The continuation of this work will be to apply the results in this
paper to obtain a characterization of algebraic functions in subvarieties
of modal algebras generated by algebras whose duals are finite line
frames.

\begin{ack}
We would like to thank the referee for reading the paper very
carefully and for his detailed and constructive criticism.
The diagrams in this paper were prepared by using the free software
\textsc{GeoGebra} \cite{gg}.
This work was  partially supported by Conicet, Secyt-UNC project
30720150100529CB, ANPCyT-PICT-2013-2011, STIC-AmSud
``Foundations of Graph Structured Data (FoG)'', and the Laboratoire
International Associ\'e ``INFINIS''.%
\end{ack}


%
%
%
%
%
%
%
%
\begin{small}\end{small}

\bigskip

\begin{small}
  \begin{quote}
    CIEM --- Facultad de Matem\'atica, Astronom\'{\i}a y F\'{\i}sica 
    (Fa.M.A.F.) 
    
    Universidad Nacional de C\'ordoba --- Ciudad Universitaria

    C\'ordoba 5000. Argentina.
  \end{quote}
\end{small}

\end{document}